\documentclass{amsart}
\usepackage{nz_style}
\title[Congruences of Eisenstein series of level $\Gamma_1(N)$]{Congruences of Eisenstein series of level $\Gamma_1(N)$ \protect\\via Dieudonn\'e theory of formal groups}
\author{Ningchuan Zhang}
\address{Department of Mathematics, Indiana University Bloomington, Bloomington, IN 47405, USA}
\hemail{nz7@iu.edu}
\date{}

\begin{document}
	\begin{abstract}
		In this paper, we give a new explanation of congruences of Eisenstein series of level $\Gamma_1(N)$ and character $\chi$. Our approach is based on Katz's algebro-geometric explanation of $p$-adic congruences of normalized Eisenstein series $E_{2k}$ of level $1$. One crucial step in our argument is to reformulate a Riemann-Hilbert correspondence in Katz's explanation in terms of Dieudonn\'e theory of height $1$ formal $A$-modules and their finite subgroup schemes. We give a  generalization of this Riemann-Hilbert correspondence in terms of formal groups of height greater than $1$. 
	\end{abstract}
	\maketitle
	\tableofcontents
	
 	In \cite{padic}, Katz gave an algebro-geometric explanation of the $p$-adic congruences of normalized Eisenstein series $E_{2k}$ of weight $2k$ and level $1$. Using a Riemann-Hilbert type correspondence (\Cref{thm:padic41}) and a theorem of Igusa, Katz showed:
	\begin{thm*}\textup{\cite[Corollary 4.4.1]{padic}}  The followings are equivalent:
		\begin{enumerate}
			\item There is a modular form of weight $2k$ whose $q$-expansion is congruent to $1$ modulo $p^m$.
			\item The integer $2k$ is divisible by $(p-1)p^{m-1}$ if $p>2$ and by $2^{\alpha(m)}$ if $p=2$, where $\alpha(1)=0,\alpha(2)=1$, and $\alpha(m)=m-2$ if $m\ge 2$.
		\end{enumerate}
	\end{thm*}
	In proving this theorem, Katz essentially used the fact that (2) is equivalent to: 
	\begin{enumerate}
		\item[(2')] \emph{The $2k$-th power representation $\Zp^{\otimes 2k}$ of $\Zpx$ is trivial mod $p^m$.}
 	\end{enumerate}
	As the space of modular forms of level $1$ and weight $2k$ are spanned by cusp forms and the normalized Eisenstein series 
	\begin{equation*}
			E_{2k}=1-\frac{4k}{B_{2k}}\sum_{n=1}^\infty \sigma_{2k-1}(n)q^n,
	\end{equation*} 
	Katz's theorem gives an upper bound of the $p$-adic valuation of $E_{2k}-1$.  The theorem of Clausen and Von Staudt on Bernoulli numbers then implies $E_{2k}$ does realize the congruence of modular forms of level $1$ and weight $2k$ predicted by that of the $\Zpx$-representation $\Zp^{\otimes 2k}$. The first goal of this paper is to adapt Katz's method to study congruences between modular forms in 
	\begin{equation*}
	M_k(\Gamma_1(N),\chi)=M_k(\Gamma_1(N))^{\chi^{-1}}, 
	\end{equation*}
	 and the constant function $1$, where $\chi\colon\znx\to \Cx$ is a primitive Dirichlet character of conductor $N$. By the $q$-expansion principle, such a modular form cannot be a cusp form and thus must have an Eisenstein series as a summand.  Our strategy is to study a $p$-adic version of this problem and then assemble the congruences at each prime. As we will be working integrally and $p$-adically, it is necessary to specify the meanings of level structures. Let $\Mell(\mu_N)$ be the stack over $\Z$ whose $R$ points are:
	\begin{equation*}
	\Mell(\mu_N)(R)=\left\{(C/R,\eta\colon\mu_N\hookrightarrow C)\left\mid \begin{array}{c}
	C\text{ is an elliptic curve over }R,\\ \eta\text{ is an embedding of group schemes}
	\end{array}
	\right.\right\}.
	\end{equation*}
	When $R$ contains a primitive $N$-th root of unity, a $\mu_N$-level structure on an elliptic curve $C$ is (non-canonically) equivalent to a classical $\Gamma_1(N)$-level structure $\underline{\Z/N}\hookrightarrow C$ on $C$. Write $N=p^vN'$ where $p$ is coprime to $N'$. The $p$-adic version of $\Mell(\mu_N)$ we will consider is $\Mello(p^v,\Gamma_1(N'))$, whose $R$ points are
	\begin{equation*}
	\Mello(p^v,\Gamma_1(N'))(R)=\left\{(C/R,\eta_p,\eta')\left\mid \begin{array}{c}
	C\text{ is a $p$-ordinary elliptic curve over }R,\\\eta_p\colon \mu_{p^{v}}\simto \wh{C}[p^v],\quad \eta'\colon \Z/N'\hookrightarrow C[N]
	\end{array}
	\right. \right\},
	\end{equation*}
	where $\wh{C}$ is the formal group of the elliptic curve $C$. The stack $\Mello(p^v,\Gamma_1(N'))$ is equivalent to the $p$-completion of $\Mell(\mu_N)$ when $p$ divides $N$, and is an open substack otherwise. Now let $\chi\colon\znx\to\Cpx$ be a $p$-adic primitive Dirichlet character of conductor $N$. Write $\Zp[\chi]=\Zp[\imag \chi]$. The character uniquely factors as product $\chi=\chi_p\cdot \chi'$, where $\chi_p$ and $\chi'$ have conductors $p^v$ and $N'$, respectively. Let $k$ be an integer such that $(-1)^k=\chi(-1)$. Denote by $\Zp^{\otimes k}[\chi]$ the $\Zpx\x\zx{N'}$-representation associated to the character 
	\begin{equation*}
		\Zpx\x\zx{N'}\xrightarrow{(a,b)\mapsto a^k\cdot\chi_p(a)\cdot\chi'(b)} (\Zp[\chi])^\x.
	\end{equation*}
	The first main result of this paper is:
	\begin{thm*}[Main Theorem \ref{thm:main}]
		Let $\Ical\trianglelefteq\Zp[\chi]$ be an ideal and $k\ge 3$. The followings are equivalent:
		\begin{enumerate}[label=\textup{(\roman*).}]
			\item There is a modular form $f$ over the stack $\Mello(p^v,\Gamma_1(N'))$ of weight $k$ and type $\chi$, with $q$-expansion $f(q)\in 1+\Ical q\llb q\rrb$.\setcounter{enumi}{4}
			\item The $\Zpx\x\zx{N'}$-representation $\Zp^{\otimes k}[\chi]$ is trivial modulo $\Ical$.
		\end{enumerate}
	\end{thm*}
	The proof of the Main Theorem has three major steps:
	\begin{enumerate}[label=\Roman*., leftmargin=*]
		\item Identify the Dirichlet character $\chi$ with the Galois descent data of a formal $\Zp[\chi]$-module $\wh{C}^{k,\chi}$ over $\Mello(\Gamma_0(N'))$, whose $R$-points are 
		\begin{equation*}
				\Mello(\Gamma_0(N'))(R)=\left\{(C/R,H\subseteq C[N'])\left\mid \begin{array}{c}
				C\text{ is a $p$-ordinary elliptic curve over }R,\\ \underline{\Z/N'}\cong H\subseteq C[N] \text{ is a flat subgroup scheme}
			\end{array}
			\right. \right\}.
		\end{equation*}
		This allows us to translate congruences of modular forms in $H^0(\Mello(p^v,\Gamma_1(N')),\bfo{k}\otimes\Zp[\chi])^{\chi^{-1}}$ to those of elements in the Dieudonn\'e module $\Dbb(\wh{C}^{k,\chi})$ of $\wh{C}^{k,\chi}$.
		\item Reformulate a Riemann-Hilbert correspondence in Katz's explanation in terms of the Dieudonn\'e module and the Galois descent data of height $1$ formal $A$-modules. In \Cref{thm:rhc_formal_A-mod}, we first relate congruences of generators in $\Dbb(\wh{C}^{k,\chi})$ to those of that of finite subgroup schemes of $\wh{C}^{k,\chi}$. The latter is then connected to congruences of the Galois representation $[\rho^{k,\chi}]$ attached to $\wh{C}^{k,\chi}$ via Galois descent. In \Cref{thm:rhc} and \Cref{thm:rhc_cong}, we give a generalization of this correspondence in terms of formal groups of heights greater than $1$. 
		\item  Factor the character $\rho^{k,\chi}$ associated to the Galois representation $[\rho^{k,\chi}]$ and then use a relative version of Igusa's theorem to reduce the group to $\Zpx\x\zx{N'}$.
	\end{enumerate}	
	The implication from (i) to (v) also follows from \cite[Lemma 1.2.2]{Katz_Eisenstein_measure} when $p$ does not divide $N$. Our method therefore gives a new explanation of the connection between congruences of modular forms and $p$-adic representations, via the Dieudonn\'e theory of formal $A$-modules. 
	
	Let $\Ical_\text{(i)}, \Ical_\text{(v)}\trianglelefteq \Zp[\chi]$ be smallest ideal satisfying (i) and (v), respectively. \Cref{thm:main} implies $\Ical_\text{(i)}=\Ical_\text{(v)}$. As the smallest ideal has generators with the largest $p$-adic valuations, we will call these smallest ideals "the maximal congruences" in both scenarios. The maximal congruences of the $\Zpx\x\zx{N'}$-representations $\Zp^{\otimes k}[\chi]$ are easy to compute, since the group is topologically finitely generated. The result of this computation is recorded in \Cref{thm:rep_congruence}.  Following this, we are particularly interested to find explicit formulas of modular forms in the Eisenstein subspace $\Ecal_{k}(p^v,\Gamma_1(N'),\chi)$ that match the congruences of $\Zp^{\otimes k}[\chi]$. Write $\chi=\chi_p\cdot\chi'$ as above. When $|\imag \chi'|$ is \emph{not} a power of $p$ or $|\imag \chi'|=1$, the maximal congruence of modular forms in the $H^0(\Mello(p^v,\Gamma_1(N')),\bfo{k}\otimes\Zp[\chi])^{\chi^{-1}}$ is realized by:
	\begin{equation*}
	E_{k,\chi}(q)=1-\frac{2k}{B_{k,\chi}}\sum_{n=1}^{\infty}\sigma_{k-1,\chi}(n)q^n, \text{ where }\sigma_{m,\chi}(n)=\sum_{0<d\mid n}\chi(d)d^m.
	\end{equation*}
	This is in particular an Eisenstein series. Here $B_{k,\chi}$ is the $k$-th \textbf{generalized Bernoulli number} associated to the Dirichlet character $\chi$. They are defined to as the Taylor coefficients of the following function:
	\[ F_\chi(t)=\sum_{a=1}^{N}\frac{\chi(a)te^{at}}{e^{Nt}-1}=\sum_{n=0}^{\infty}B_{k,\chi}\frac{t^k}{k!},\qquad \chi(a)=0\text{ if }(a,N)\ne 1.\]
	Arithmetic properties of $B_{k,\chi}$ were studied in \cite{Carlitz_GBN}. The argument in our paper therefore relates the \emph{denominator} of $\frac{B_{k,\chi}}{2k}$ to congruences of the $\Zpx\x\zx{N'}$-representations $\Zp^{\otimes k}[\chi]$ in this case.
	
	When $|\imag \chi'|>1$ is a power of $p$, the maximal congruence is realized as a linear combination of $E_{k,\chi}$ with some other modular forms in $H^0(\Mello(p^v,\Gamma_1(N')),\bfo{k}\otimes\Zp[\chi])^{\chi^{-1}}$. In this case, congruences of the representation $\Zp^{\otimes k}[\chi]$ shed light on the \emph{numerator} of $\frac{B_{k,\chi}}{2k}$. One such example is:
	\begin{cor*}[\ref{prop:numerator_valuation} and \ref{cor:cong_gp_coh}]
		Let $p>2$ be a prime and $\chi\colon\zx{\ell}\to \Cpx$ be a Dirichlet character of conductor $\ell$ such that $\ell\neq p$ is a prime number and  $|\imag \chi'|=|\imag \chi|$ is a $p$-power. Denote the maximal ideal of $\Zp[\chi]$ by $\mfrak$. Assume $(-1)^k=\chi(-1)$. Then $\frac{B_{k,\chi}}{2k}$ is an algebraic $p$-adic integer in $\Zp[\chi]$ by \textup{\cite[Theorem 1]{Carlitz_GBN}}. We then have
		\begin{equation*}
		\frac{B_{k,\chi}}{2k}\in \mfrak, \text{ when } (p-1)\nmid k.
		\end{equation*} 
		This relation is reflected in the fact that the maximal congruence of the $\Zpx\x\zx{N'}$-representation $\Zp^{\otimes k}[\chi]$ is $(1)$ when $(p-1)\nmid k$.
	\end{cor*}
	We further note that congruences of $p$-adic representations of $\Zpx\x\zx{N'}$ are related to its group cohomology.
	\begin{cor*}[\ref{cor:cong_gp_coh}] Let $\Ical\trianglelefteq \Zp[\chi]$ be an ideal. The followings are equivalent:
		\begin{enumerate}
			\item The ideal $\Ical$ is the maximal congruence of modular forms in $H^0(\Mello(\Gamma_0(N')),\bfomega^{k,\chi})$.
			\item The ideal $\Ical$ is the maximal congruence of $\Zp^{\otimes k}[\chi]$ as a trivial $\Zpx\x\zx{N'}$-representation. 
			\item $H^1_c(\Zpx\x\zx{N'};\Zp^{\otimes k}[\chi])\cong  \Zp[\chi]/\Ical$.
		\end{enumerate}
	\end{cor*}
	Combined with the explicit formulas of the Eisenstein series that realizes the maximal congruences, \Cref{cor:cong_gp_coh} implies that the group cohomology $H^1_c(\Zpx\x\zx{N'};\Zp^{\otimes k}[\chi])$ computes the denominator of $\frac{B_{k,\chi}}{2k}$ when $|\imag \chi'|$ is not a power of $p$ and $(-1)^k=\chi(-1)$. We conclude this paper by noting that when the character $\chi$ is trivial, the continuous group cohomology $H^1_c(\Zpx;\Zp^{\otimes 2k})$ appears in other fields of mathematics:
	\begin{itemize}
		\item In chromatic homotopy theory, this group cohomology computes the $p$-primary part of the image of the $J$-homomorphism in the stable homotopy groups of the sphere. In this way, we have given a new explanation of the connection between congruences of the normalized Eisenstein series $E_{2k}$ and the image of $J$.
		\item In algebraic $K$-theory, a theorem of Soul\'e implies this group cohomology is isomorphic to certain \'etale cohomology which appears in the Lichtenbaum Conjecture for the Riemann $\zeta$-function. 
	\end{itemize}

	\subsection*{Notations and conventions}
	\begin{itemize}
		\item Denote the \Teichmuller character by the Greek letter $\omega$  and denote the sheaf of invariant differentials on various stacks by the boldface version of the same Greek letter $\bfomega$. 
		\item Write $\Cp$ for the analytic completion of $\overline{\Qp}$, the algebraic closure of the rational $p$-adics.
		\item Write $\underline{G}$ for the constant $G$-group scheme.
		\item Write $\Gah$ and $\Gmh$ for the additive and multiplicative formal groups, respectively. Denote by $\mu_N$ the $N$-torsion subgroup scheme of $\Gmh$, and by $\alpha_p$ the kernel of the $p$-th power isogeny of $\Gah$ over an $\Fp$-algebra.
		\item By a height $1$ or slope $1$ formal group $\Gh$, we mean $\Gh$ is \'etale locally isomorphic $\Gmh^{\oplus d}$, where $d$ is the dimension of $\Gh$. 
		\item Let $M$ be a $G$-representation in an $R$-modules and $\chi\colon G\to R^\x$ be a character. We write $M^\chi$ for the $\chi$-eigensubspace of $M$.
		\item We will suppress the $\Zp$ in $M \otimes_{\Zp}  N$  when $M$ and $N$ are both $\Zp$-modules.
		\item Let $\chi$ be a Dirichlet character of conductor $N$. Write $N=p^vN'$, where $p\nmid N'$. Then there is a unique decomposition $\chi=\chi_p\chi'$, where the conductors of $\chi_p$ and $\chi'$ are $p^v$ and $N'$, respectively. We fix the meanings of $N$, $N'$, $v$, $\chi_p$, and $\chi'$ throughout the paper.
		\item We will write "$\cong$" for 1-categorical isomorphisms and "$\simeq$" for equivalences of categories or stacks. 
	\end{itemize}
	\subsection*{Acknowledgments}
	I would like to thank Matt Ando for advising me to think about the implications of \cite[Chapter 4]{padic} in homotopy theory, which eventually leads to this paper; Patrick Allen for patiently answering my many questions on modular forms and moduli of elliptic curves, and for correcting a few mistakes I made in an earlier version of this paper. I would also like to thank Mark Behrens, Dominic Culver, Elden Elmanto, Jacklyn Lang, Charles Rezk, Shiyu Shen, and Vesna Stojanoska for many helpful discussions and comments. Finally, I would like to thank the anonymous referees for many helpful comments and suggestions on revisions.
	
	N. Zhang was partially supported by NSF grant DMS-2304719, during the revision of this work. 
	\section{$\mu_N$-level structures on elliptic curves and modular forms}
	\subsection{The Eisenstein subspace}
	Let $\chi\colon \znx\to \Cx$ be a primitive Dirichlet character of conductor $N$. We are now going to introduce the Eisenstein series of level $\Gamma_1(N)$ and character $\chi$, following \cite[\S 5.1]{Hida_Eisenstein} and \cite[Chapter 5]{Stein_modular_forms}.  
	\begin{defn}
		Let $\Gamma\le\SL_2(\Z)$ be a congruence subgroup. Let $\mathbb{T}\subseteq \mathrm{End}(M_k(\Gamma))$ be the Hecke algebra acting on $M_k(\Gamma)$. Then there is decomposition of $\mathbb{T}$-modules: 
		\begin{equation}\label{eqn:Tmod_decomp}
		M_k(\Gamma)=\Ecal_{k}(\Gamma)\oplus \mathcal{S}_k(\Gamma),
		\end{equation}
		where $\mathcal{S}_k(\Gamma)$ is subspace of cusp forms, i.e. modular forms that vanish at \emph{all} cusps. The subspace $\Ecal_{k}(\Gamma)$ is the \textbf{Eisenstein subspace} of weight $k$ and level $\Gamma$.
	\end{defn}
	\begin{exmp}
		Below is a family of Eisenstein series in $\Ecal_{k}(\Gamma_1(N),\chi)$. Let $\chi_1\colon\zx{N_1}\to \Cx$ and $\chi_2\colon\zx{N_2}\to\Cx$ be two primitive Dirichlet characters of conductors $N_1$ and $N_2$. Define an Eisenstein series:
		\begin{equation*}
		G_{k,\chi_1,\chi_2}(z)=\sum_{(n,m)\ne (0,0)}\frac{\chi_1(m)\chi_2^{-1}(n)}{(mNz+n)^k}.
		\end{equation*}
		This is an Eisenstein series of weight $k$ and level $N_1N_2$.
	\end{exmp}
	\begin{thm}[{\cite[Theorem 4.5.2]{Diamond-Shurman}}]\label{thm:eisenstein_subspace_basis}
		Let $N>1$ be a positive integer and $k\ge 3$. The Eisenstein series $\{G_{k,\chi_1,\chi_2}(tz)\mid (N_1N_2t)|N, \chi_2/\chi_1=\chi\}$ forms a basis of $\Ecal_{k}(\Gamma_1(N),\chi)$.
	\end{thm}
	\subsection{$\mu_N$-level structures}
	As we will be working integrally and $p$-adically at levels divisible by $p$, it is necessary to specify the meaning of $\Gamma_1(N)$-level structures.
	\begin{defn}
		A $\mu_N$-level structure on an elliptic curve $C$ is an embedding of group schemes $\eta\colon\mu_N\hookrightarrow C$. Denote by $\Mell(\mu_N)$ the moduli stack of elliptic curves with $\mu_N$-level structures. Let $R$ be a ring. The $R$ points of $\Mell(\mu_N)$ are 
		\begin{equation*}
		\Mell(\mu_N)(R)=\left\{(C/R,\eta)\left\mid \begin{array}{c}
		C\text{ is an elliptic curve over $R$ and }\\\eta\colon\mu_N\hookrightarrow C\text{ is an embedding of group schemes}
		\end{array}\right.\right\}.
		\end{equation*}
		Define the space of modular forms of weight $k$ and level $\mu_N$ by
		\begin{equation*}
		M_k(\mu_N)=H^0(\Mell(\mu_N),\bfo{k}),\quad M_k(\mu_N,\chi)=M_k(\mu_N)^{\chi^{-1}},
		\end{equation*}
		where $\chi$ is a Dirichlet character of conductor $N$.
	\end{defn}	
	\begin{lem}\label{lem:mu_N_Gamma_1_N}
		The stacks $M_k(\Gamma_1(N),\chi)$ and $M_k(\mu_N,\chi)$ are equivalent over $\Cbb$.
	\end{lem}
	\begin{proof}
		This is because $\Mell(\Gamma_1(N))(R)\simeq \Mell(\mu_N)(R)$ when $R$ contains a primitive $N$-th root of unity.
	\end{proof}
	\begin{prop}\label{prop:mu_N_rigid}
		When $N\ge 4$, $\Mell(\mu_N)$ is represented by a smooth affine curve over $\Z$.
	\end{prop}
	\begin{proof}
		By \cite[Corollary 4.7.1]{amec}, it suffices to show: 
		\begin{enumerate}
			\item The forgetful map $\Mell(\mu_N)\to \Mell$ is relatively representable, affine, and \'{e}tale.
			\item $\Mell(\mu_N)$ is rigid, meaning that there is no non-trivial automorphism of the pair $(C,\eta\colon\mu_N\hookrightarrow C)$.
		\end{enumerate} 
		(1) is proved in \cite[Section 4.9, 4.10]{amec}. (2) is proved in the \cite[Corollary 2.7.4]{amec} when $N\ge 4$.
	\end{proof}	
	
	\subsection{The $q$-expansion principle}
	Let $\Mell(\Gamma)_R$ be moduli stack of generalized elliptic curves over $R$-schemes with $\Gamma$-level structures. 
	\begin{defn}
		A \textbf{cusp} in $\Mell(\Gamma)_R$ is an embedding $\spf R\llb q\rrb\to\Mell(\Gamma)_R$ that classifies a $\Gamma$-level structure on the Tate curve $T(q)$. The \textbf{$q$-expansion} of a modular form $f\in H^0(\Mell(\Gamma)_R,\bfo{k})$ at a cusp is its image under restriction map to the said cusp.
	\end{defn}
	\begin{prop}[The $q$-expansion principle,{ \cite[Theorem 1.6.1]{padic}}]\label{prop:q-exp}
		Let $f$ be a modular form of weight $k$, level $\Gamma$, and coefficients in $R$. It is zero iff its restrictions to all cusps are zero. Furthermore, when the stack $\Mell(\Gamma)_R$ is irreducible, the restriction map to \emph{any} cusp is injective.
	\end{prop}
	It follows that congruences of modular forms are determined by their $q$-expansions at any cusp when $\Mell(\Gamma)_R$ is irreducible. By \cite[Theorem 1.2.1]{Conrad_arithmetic_moduli}, this is indeed the case when $\Gamma=\Gamma_1(N)$ and $R=\Z$.
	
	Now normalize $E_{k,\chi_1,\chi_2}$ so that its coefficients are algebraic integers.
	\begin{defn}[Normalization of $G_{k,\chi_1,\chi_2}$]\label{defn:E_k_normalization} When $\chi_2$ is non-trivial, we define normalized Eisenstein series:
		\begin{equation*}
		E_{k,\chi_1,\chi_2}(q)=\sum_{n\ge 1}\left(\sum_{0<d\mid n}\chi_2(d)\chi_1(n/d)d^{k-1}\right)q^n.
		\end{equation*}
		When $\chi_1$ is the trivial character $\chi^0$ and $\chi_2=\chi$, we define $E_{k,\chi}$ and $E_{k,\chi^0,\chi}$ by 
		\begin{align*}
		E_{k,\chi}(q)&=1-\frac{2k}{B_{k,\chi}}\sum_{n\ge 1}\left(\sum_{0<d\mid n}\chi(d)d^{k-1}\right)q^n\\		
		E_{k,\chi^0,\chi}(q)=c\cdot E_{k,\chi}(q)&=c_0+c_1\sum_{n\ge 1}\left(\sum_{0<d\mid n}\chi(d)d^{k-1}\right)q^n,\quad c_0, c_1\in \Z[\chi] \text{ are coprime and } c_0/c_1=-\frac{B_{k,\chi}}{2k}.
		\end{align*}
	\end{defn}
	\begin{rem}
		As $\Z[\chi]$ has non-trivial unit group, the constant $c$ is not unique in general. 
	\end{rem}
	\begin{prop}
		$E_{k,\chi_1,\chi_2}(q)\in (H^0(\Mell(\mu_N),\bfo{k})\otimes_\Z \Z[\chi_1,\chi_2])^{\chi_1/\chi_2}$.
	\end{prop}
	\begin{proof}
		By \Cref{lem:mu_N_Gamma_1_N}, $E_{k,\chi_1,\chi_2}\in M_k(\mu_N)$. It is in the $\chi_1/\chi_2$-eigensubspace by \Cref{thm:eisenstein_subspace_basis}. As the coefficients of $E_{k,\chi_1,\chi_2}(q)$ are all in $\Z[\chi_1,\chi_2]$ by \Cref{defn:E_k_normalization}, the $q$-expansion principle \Cref{prop:q-exp} implies that 
		\begin{equation*}
		E_{k,\chi_1,\chi_2}\in H^0(\Mell(\mu_N)\x_{\Z} \spec \Z[\chi_1,\chi_2],\bfo{k}).
		\end{equation*}
		When the conductors of $\chi_1$ and $\chi_2$ are $3$, their images are $\{\pm 1\}$ and $\Z[\chi_1,\chi_2]=\Z$. When the conductors of $\chi_1$ and $\chi_2$ are at least $4$, the claim follows from \Cref{prop:mu_N_rigid}.
	\end{proof}
	\subsection{$p$-adic modulis}
	We will study congruences of modular forms in  $M_k(\mu_N,\chi)$ completed at a prime $p$.
	\begin{defn}
		An elliptic curve $C$ over a $p$-complete ring is called ($p$-)\textbf{ordinary} if it has nodal singularity, or its reduction mod $p$ is ordinary, i.e. the formal group $\widehat{C}$ associated to $C$ has height $1$ reduction mod $p$.
		
		Denote the $p$-completed moduli stack of $p$-ordinary elliptic curve by $\Mello$. This is an open substack of $\Mell$, since it is the non-vanishing locus of the Hasse invariant.
	\end{defn}
	
	Restricted to $\Mello$, the $\mu_{p^{v}}$-level structures on an elliptic curve $C$ are identified with the corresponding level structures on the height $1$ formal group $\widehat{C}$. As formal groups of height $1$ are \'{e}tale locally isomorphic to $\Gmh$, the multiplicative formal group, there is a tower of stacks:
	\begin{equation*}
	\begin{tikzcd}
	\Mell^{triv}\rar&\cdots\rar& \Mello(p^2)\rar&\Mello(p)\rar&\Mello,
	\end{tikzcd}
	\end{equation*}
	where $\Mello(p^v)$ and $\Mell^{triv}$ are the moduli stacks with $R$-points $(C,\eta\colon \mu_{p^v}\simto \widehat{C}[p^v])$ and $(C,\eta\colon\Gmh\simto \widehat{C})$ respectively. The forgetful map $\Mello(p^v)\to \Mello$ is a $\zx{p^v}$-torsor and $\Mell^{triv}\to \Mello$ is a $\Zpx$-torsor. There is a pullback diagram of towers of stacks:
	\begin{equation}\label{igusa_tower}
	\begin{tikzcd}
	\Mell^{triv}\rar\dar\arrow[dr, phantom, "\lrcorner", very near start]&\cdots\dar\rar\arrow[dr, phantom, "\lrcorner",  near start]& \Mello(p^2)\arrow[dr, phantom, "\lrcorner", very near start]\rar\dar&\Mello(p)\arrow[dr, phantom, "\lrcorner", very near start]\rar\dar&\Mello\dar\\
	\spf \Zp\rar&\cdots\rar& B(1+p^2\Zp)\rar&B(1+p\Zp)\rar&B\Zpx
	\end{tikzcd}
	\end{equation}
	
	\begin{prop}\textup{\cite{padicL,ctmf}}\label{prop:mello_p_affine}
		When $p>2$ or $p=2$ and $v>1$, $\Mello(p^v)$ and $\Mell^\text{triv}$ are affine formal schemes. In particular, $\Mell^{triv}\simeq \spf D_p$ where $D_p$ is the ring of divided congruences of $p$-adic modular forms. 
	\end{prop}
	
	The strategy now is to relate congruences of $E_{k,\chi}$ to finite subgroups of the formal groups and formal $A$-modules associated to $p$-ordinary elliptic curves. Below are some facts about needed in the study of formal group of a $p$-ordinary elliptic curve. 
	\begin{prop}\label{prop:can_subgp}
		Let $C$ be a $p$-ordinary elliptic curve over a $\Zp$-algebra. Denote its formal group by $\widehat{C}$.
		\begin{enumerate}
			\item $C$ has a canonical subgroup $H$ of order $p$, where $H=\widehat{C}[p]$.
			\item The quotient map $\varphi\colon  C\mapsto C/H$ is the relative Frobenius map on $\Mello$.
			\item Let $f(q)$ be the $q$-expansion of a modular form over $\Mello$, then $\varphi^*f(q)=f(q^p)$.
			\item There is an isomorphism of invertible sheaves $F\colon \varphi^*\bfomega\simto \bfomega$ over $\Mello$, where $\bfomega$ is the sheaf of invariant differentials of $C$.
		\end{enumerate} 
	\end{prop}
	We conclude by comparing the integral and $p$-adic moduli problems.
	\begin{lem}\label{lem:mu_N_p_ord}
		If an elliptic curve $C$ admits a $\mu_N$-level structure, then it is $p$-ordinary for all primes $p\mid N$.
	\end{lem}
	\begin{proof}
		As $\mu_p$ is a subgroup scheme of $\mu_N$ when $p\mid N$, it suffices to prove the case when $N=p$. Notice $\mu_p$ is $p$-torsion, any embedding of $\mu_p$ into an elliptic curve $C$ must factor through $C[p]$. When $C$ is $p$-supersingular, $C[p]=\wh{C}[p]$. Thus it reduces to showing that there is no embedding of $\mu_p$ into a height $2$ formal group. 
		
		Using Dieudonn\'e theory of finite groups schemes, we can show the only finite subgroup scheme of rank $p$ in a height $2$ formal group is \'etale locally isomorphic to $\alpha_p$, which is not \'etale locally isomorphic to $\mu_p$.
	\end{proof}
	\begin{defn}
		Let $\Mello(p^v,\Gamma_1(N'))$ be the stack whose $R$-points are
		\begin{equation*}
		\Mello(p^v,\Gamma_1(N'))(R)=\left\{(C/R,\eta_p,\eta')\left\mid \begin{array}{c}
		C\text{ is a $p$-ordinary elliptic curve over }R,\\\eta_p\colon \mu_{p^{v}}\simto \wh{C}[p^v],\quad \eta'\colon \Z/N'\hookrightarrow C[N]
		\end{array}
		\right. \right\}.
		\end{equation*}
	\end{defn}
	\begin{prop}\label{prop:mu_N_p-completion}
		Write $N=p^v\cdot N'$, where $p\nmid N'$. Then we have
		\begin{equation*}
		(\Mell(\mu_N))^\wedge_{p}\simeq \left\{\begin{array}{cl}
		\Mello(p^v,\Gamma_1(N')),& \textup{if } p\mid N;\\
		(\Mell)^\wedge_{p}(\Gamma_1(N)),& \textup{if } p\nmid N.
		\end{array}\right.
		\end{equation*}
	\end{prop}
	\begin{proof}
		This follows from \Cref{lem:mu_N_p_ord}.
	\end{proof}
	\begin{prop}\label{prop:xi_torsor_stacks}
		The forgetful map $\xi:\Mello(p^v,\Gamma_1(N'))\to \Mello(\Gamma_0(N'))$ is a $\znx$-torsor of stacks.
	\end{prop}
	\begin{proof}
		One can check this by unraveling the definition of $G$-torsors for stacks.		
	\end{proof}
	\begin{prop}\label{prop:mu_3_rigid}
		The stack $\Mello(p^v,\Gamma_1(N'))$ is represented by a smooth formal affine curve over $\Zp$ in the following cases:
		\begin{itemize}
			\item $N=p^v\cdot N'\ge 4$ for any $p$.
			\item $N=p=3$.
			\item $N=N'=3$ and $p\equiv 2\mod 3$.
		\end{itemize}
	\end{prop}
	\begin{proof}
		When $N\ge 4$, the stack $\Mello(p^v,\Gamma_1(N'))$ is the $p$-completion (when $p\mid N$), or a distinguished open substack of the $p$-completion (when $p\nmid N$) of $\Mell(\mu_N)$ by \Cref{prop:mu_N_p-completion}. As the latter is represented by a smooth affine curve over $\Z$ by \Cref{prop:mu_N_rigid}, the first case of the claim follows.
		
		When $N=p=3$, $\Mello(3)$ is affine by \Cref{prop:mello_p_affine}.
		
		When $N=3$ and $p\neq 3$, it suffices to show the moduli problem is rigid as in the proof of \Cref{prop:mu_N_rigid}. Let $\varepsilon$ be a nontrivial automorphism of $C$ that preserves a $\Gamma_1(3)$-level structure $\eta'\colon \underline{\Z/3}\hookrightarrow C[3]$. Adapting the proof of \cite[Corollary 2.7.3]{amec} to the $N=3$ case, we can show $\varepsilon$ must satisfy $\varepsilon^2+\varepsilon+1=0$. This implies $\varepsilon$ is an element of order $3$ in $\aut(C)$. By \cite[Proposition A.1.2.(c)]{aec}, $\aut(C)$ has an element of order $3$ iff the $j$-invariant of the elliptic curve $C$ is $0$. By \cite[Example V.4.4, Exercise 5.7]{aec}, the $j=0$ elliptic curve is $p$-supersingular when $p\equiv 2\mod 3$. As a result, when $p\equiv 2\mod 3$, there is no non-trivial automorphism of a $p$-ordinary elliptic $C$ that preserves a $\Gamma_1(3)$-structure. This shows the moduli problem $\Mello(\Gamma_1(3))$ is rigid at such primes, and hence represented by a smooth formal affine curve over $\Zp$.
	\end{proof}
	\begin{rem}
		The moduli problem $\Mello(\Gamma_1(3))$ is \emph{not} rigid when $p\equiv 1\mod 3$. For such primes, the $j=0$ elliptic curve $C$ is $p$-ordinary.  $C$ has an automorphism $\varepsilon$ of order $3$. As $C[3]$ is isomorphic to the constant group scheme $\underline{\Z/3^{\oplus 2}}$, the automorphism $\varepsilon$ restricts to an element of order $3$ in $\GL_2(\Z/3)$. From the identity $0=\varepsilon^3-1=(\varepsilon-1)^3$ in $\mathrm{End}(C[3])\cong  M_2(\Z/3)$, $\varepsilon$ is unipotent. Then there is a basis $\{P,Q\}$ of $C[3]$ under which $\varepsilon$ acts by the matrix $\begin{pmatrix}
		1&1\\0&1
		\end{pmatrix}$. Let $\eta'\colon \underline{\Z/3}\hookrightarrow C[3]$ that sends $1\in \underline{\Z/3}$ to $P\in C[3]$. The matrix representations of $\varepsilon$ shows it is an automorphism of the pair $(C,\eta')$. Consequently, $\Mello(\Gamma_1(3))$ is not rigid. Therefore the moduli problem not represented by a scheme.
	\end{rem}
	\begin{prop}\label{prop:Eisenstein_decomp}
		Let $\chi$ be a Dirichlet character of conductor $N$, where $N=p^vN'$ with $p\nmid N'$. Denote the Eisenstein subspace in the $\chi^{-1}$-eigensubspace in $H^0(\Mello(p^v,\Gamma_1(N')),\bfo{k}\otimes\Zp[\chi])$ by $\Ecal_{k}(p^v,\Gamma_1(N'),\chi)$. Then we have a decomposition:
		\begin{equation*}
		\Ecal_k(\mu_N,\chi)^\wedge_p\cong  \bigoplus_{[\sigma]\in \cok \iota^*} \Ecal_k(p^v,\Gamma_1(N'),\iota\circ \sigma\circ \chi),
		\end{equation*}
		where $\iota\colon \Q(\chi)\hookrightarrow \Cp$ is a field extension and $\iota^*\colon  \gal(\iota(\Q(\chi))/\Qp)\to \gal(\Q(\chi)/\Q)$ is the induced map of $\iota$ on Galois groups.
	\end{prop}
	\begin{proof}
		This is a result of the equivalence of $p$-adic $\znx$-representations \cite[Corollary A.3.5]{nz_Dirichlet_J}:
		\begin{equation*}
		\Z[\chi]\otimes_\Z \Zp\cong  \bigoplus_{[\sigma]\in \cok \iota^*} \Zp[\iota\circ \sigma\circ \chi].\qedhere
		\end{equation*}
	\end{proof}
	\begin{cor}\label{cor:p-adic_basis}
		Let $\chi_1$ and $\chi_2$ be $p$-adic Dirichlet characters of conductors $N_1$ and $N_2$, respectively. Then the normalized Eisenstein series $E_{k,\chi_1,\chi_2}$ in \Cref{defn:E_k_normalization} defines a $p$-adic Eisenstein series in $\Ecal_{k}(p^v,\Gamma_1(N'),\chi_2/\chi_1)$, where $N=N_1N_2=p^vN'$ and $p\nmid N'$.
	\end{cor}
	\section{Eisenstein series and Galois representations}
	In this section, we adapt Katz's explanation of congruences of $E_{2k}$ as $p$-adic modular forms in \cite{padic} to study the congruences of $p$-adic Eisenstein series with level $(\mu_{p^v},\Gamma_1(N'))$. The statement and proof of the Main Theorem (\ref{thm:main}) rely heavily on the Dieudonn\'e theory of formal groups and formal $A$-modules, which will be briefly reviewed in the next subsection. A reference for the general theory of formal groups and Dieudonn\'e theory can be found in \cite{Demazure_p-div}. 
	\subsection{Review of Dieudonn\'e modules and Galois descent of formal groups}
	Let $R$ be a $p$-complete smooth $\Zp$-algebra such that $R/p$ is an integrally closed domain and $R$ admits an endomorphism $\varphi\colon  R\to R$ that lifts the $p$-th power map on $R$.
	
	The \textbf{Dieudonn\'e module} $\Dbb(\Gh)$ of a formal group $\Gh_0$ over $R/p$ is a triple
	\begin{equation*}
	\Dbb(\Gh)=(M,F\colon \varphi^*M\longrightarrow M, V\colon M\longrightarrow \varphi^*M),
	\end{equation*}
	where $M=PH^1_{\dR}(\Gh/R)$ is the primitives in the de-Rham cohomology for some lift $\Gh$ of $\Gh_0$ to $R$ and $FV=p=VF$ on the respective domains. Formal groups of the same height $h<\infty$ over $R/p$ are \'etale locally isomorphic to each other. It follows that their isomorphism classes are classified by the continuous Galois cohomology $H_c^1(\piet(R/p);\aut(\Gamma_h))$, where $\Gamma_h$ is Honda formal group of height $h$. The Galois cohomology class $[\rho]\in H_c^1(\piet(R/p);\aut(\Gamma_h))$ that corresponds to $\Gh_0$ is called the \textbf{Galois descent data} of $\Gh_0$.
	
	When $\Gh$ has height (slope) $1$, $PH^1_{\dR}(\Gh/R)=\bfomega(\Gh)$ is the sheaf of invariant differentials of $\Gh$ and $F\colon \varphi^*M\longrightarrow M$ is an isomorphism. As a result, the Verschiebung $V$ is determined by $F$ in this case. We will write $\Dbb(\Gh)=(\bfomega(\Gh),F\colon \varphi^*\bfomega(\Gh)\simto \bfomega(\Gh))$ when $\Gh$ has height $1$.
	\begin{exmp}\label{exmp:D_Gmh}
		Let $R$ be a $p$-complete algebra and $\varphi\colon R\to R$ be a lift of Frobenius map. Denote the Dieudonn\'e module of $\Gmh/R$, the multiplicative formal group over $R$, by  $\Dbb(\Gmh)=(M,F\colon \varphi^*M\simto M)$. Then $M$ is a free $R$-module of rank $1$ generated by an element $\gamma$ such that $F(\gamma)=\gamma$.
	\end{exmp}
	The Galois descent data of height $1$ formal groups are described by the following:
	\begin{prop}\label{prop:FG_Isom_ab}
		Isomorphism classes of formal groups over a $p$-complete algebra $R$ with height $1$ reductions modulo $p$ are classified by the abelian group $\hom(\piet(R),\Zpx)$. In particular, the constant map in $\hom(\piet(R),\Zpx)$ corresponds to $\Gmh$.
	\end{prop}
	\begin{proof}
		When $h=1$, $\Gamma_1=\Gmh$ and $\aut(\Gmh)\cong \Zpx$ is an abelian group. Since $R$ is $p$-complete, we have $\piet(R)\cong  \piet(R/p)$. Using the fact that formal groups of height $1$ over $R/p$ are \'etale locally isomorphic to $\Gmh$, the group cohomology $H_c^1(\piet(R);\Zpx)\cong  H_c^1(\piet(R/p);\Zpx)$ classifies formal groups of height $1$ over $R/p$ up to isomorphisms. In particular, the Galois cohomology class represented by the constant map corresponds to $\Gmh$ over $R/p$. This Galois cohomology is an abelian group since $\Zpx$ is an abelian group. As the \'etale fundamental group acts trivially on $\Zpx$, we have $H_c^1(\piet(R);\Zpx)\cong \hom(\piet(R),\Zpx)$. This shows $\Gmh$ is classified by the constant group homomorphism in $\hom(\piet(R),\Zpx)$.
		
		By the Lubin-Tate deformation theory of formal groups, height $1$ formal groups over $R/p$ have unique deformations to $R$. This yields $\hom(\piet(R),\Zpx)\cong  H_c^1(\piet(R);\Zpx)$ classifies formal groups over $R$ with height $1$ reductions modulo $p$ up to isomorphisms.
	\end{proof}
	\Cref{prop:FG_Isom_ab} suggests a natural closed symmetric monoidal structure in the category of $1$-dimensional formal groups of height $1$. Let $\rho_i\colon \piet(R)\to\Zpx$ be the Galois descent data for the height $1$ formal groups $\Gh_i$, $i=1,2$. Then the Galois descent data for $\Gh_1\otimes \Gh_2$ is $\rho_1\cdot \rho_2$. In terms of Dieudonn\'e modules, this monoidal structure is described by
	\begin{equation*}
	\Dbb(\Gh_1\otimes \Gh_2)=(\bfomega_1\otimes_R \bfomega_2,  F_1\otimes F_2\colon \varphi^*(\bfomega_1\otimes_R \bfomega_2)\cong  \varphi^*\bfomega_1\otimes_{(\varphi^*R)} \varphi^*\bfomega_2\simto \bfomega_1\otimes_R \bfomega_2),
	\end{equation*}
	where $\Dbb(\Gh_i)=(\bfomega_i,F_i:\varphi^*\bfomega_i\simto \bfomega_i)$. Below are two relevant examples for this paper:
	
	\begin{exmp}\label{exmp:FG_ell}
		Let $C$ be the universal elliptic curve over $\Mello$ and $\wh{C}$ be its formal group. $\wh{C}$ is a height $1$ formal group since $C$ is a $p$-ordinary elliptic curve. Denote the Galois descent data for $\wh{C}$ by $\rho^1\colon \piet(\Mello)\to \Zpx$. The pair $(\bfomega, F\colon \varphi^*\bfomega\simto \bfomega)$ described in \Cref{prop:can_subgp} is the Dieudonn\'e module of $\wh{C}$. On $q$-expansions of modular forms, the Frobenius acts by the formula $F(f(q))=f(q^p)$. Denote of the $k$-th monoidal power of $\wh{C}$ by $\wh{C}^{\otimes k}$. The Galois descent data for $\wh{C}^{\otimes k}$ is 
		\begin{equation*}
		\rho^k \colon  \piet(\Mello)\xrightarrow{\rho^1}\Zpx\xrightarrow{(-)^k}\Zpx,
		\end{equation*}
		The Dieudonn\'{e} module of $\wh{C}^{\otimes k}$ is
		\begin{equation*}
		\Dbb(\wh{C}^{\otimes k})=(\bfo{k},F^{\otimes k}\colon \varphi^*\bfo{k}\simto \bfo{k}),
		\end{equation*}
		where $F^{\otimes k}(f(q))=f(q^p)$. 
	\end{exmp}
	
	As the Eisenstein series we study in this paper have coefficients in $\Zp[\chi]$, it is necessary to work with formal $\Zp[\chi]$-modules. Let $A$ be an algebra. A formal $A$-module is a formal group $\Gh$ together with an embedding of algebras $i\colon A\hookrightarrow \mathrm{End}_{FG}(\Gh)$ such that the composite
	\begin{equation*}
	\begin{tikzcd}
	A\rar[hookrightarrow]& \mathrm{End}_{FG}(\Gh)\rar& \mathrm{End}(\bfomega(\Gh))
	\end{tikzcd}	
	\end{equation*}
	realizes $\bfomega(\Gh)$ as an $A$-module. We will write the power series representation of $i(a)$ by $[a]$. Any formal group $\Gh$ comes with a unique formal $\Z$-module structure. When $\Gh$ is defined over a $p$-complete ring $R$, this formal $\Z$-module structure extends (uniquely) to a formal $\Zp$-module structure, since $\displaystyle\lim\limits_{v\to\infty}[p^v](t)=0$ in $R\llb t\rrb$.
	\begin{const}\label{con:Formal_A-mod}
		When $A$ is $\Zp$-algebra that is a finite free $\Zp$-module, we define a formal $A$-module $\Gh\otimes A$ out of a $1$-dimensional formal group $\Gh$. The underlying formal group of $\Gh\otimes A$ is $\Gh^{\oplus r}$, where $r$ is the rank of $A$ as a free $\Zp$-module. The $A$-action on $\Gh\otimes A=\Gh^{\oplus r}$ is given by 
		\begin{equation*}
		\begin{tikzcd}
		A=\mathrm{End}_{A\text{-mod}}(A)\rar[hookrightarrow]& \mathrm{End}_{\Zp}(\Zp^{\oplus r})\rar[hookrightarrow]& \mathrm{End}_{FG}(\Gh^{\oplus r}).
		\end{tikzcd}
		\end{equation*}
		where the first map is induced by $A\cong  \Zp^{\oplus r}$. Write $\Dbb(\Gh)=(\bfomega(\Gh),F,V)$. The Dieudonn\'{e} module of $\Gh\otimes A$ is 
		\begin{equation*}
		\Dbb(\Gh\otimes A)=\Dbb(\Gh)\otimes A=(\bfomega(\Gh)\otimes A, F\otimes 1,V\otimes 1).
		\end{equation*}
		When the height of $\Gh$ is $h$, let $[\rho]\in H_c^1(\piet(R);\aut(\Gamma_h))$ be the Galois descent data for $\Gh$. The formal $A$-module $\Gh\otimes A$ is \'etale locally isomorphic to $\Gamma_h\otimes A$. Notice there is an embedding of algebras:
		\begin{equation*}
		\begin{tikzcd}
		i\colon \mathrm{End}(\Gamma_h)\rar[hookrightarrow]& \mathrm{End}_{\text{formal $A$-mod}}(\Gamma_h\otimes A)\cong \mathrm{End}(\Gamma_h)\otimes A \quad g\longmapsto g\otimes 1.
		\end{tikzcd}
		\end{equation*}
		The embedding $i$ restricts to a group homomorphism on units. The Galois descent data for $\Gh\otimes A$ is then the image of $[\rho]$ under the induced map of $i$ in Galois cohomology 
		\[i_*\colon H_c^1(\piet(R);\aut(\Gamma_h))\longrightarrow H_c^1(\piet(R);\aut_{\text{formal $A$-mod}}(\Gamma_h\otimes A)).\] 
	\end{const}	
	\subsection{Statement of the Main Theorem}\label{subsec:sketch_proof}
	Let $\chi\colon \znx\to\Cpx$ be a Dirichlet character of conductor $N$. Write $N=p^vN'$, where $p\nmid N'$. Then $\chi$ uniquely factors as a product $\chi=\chi_p\cdot \chi'$, where $\chi_p$ and $\chi'$ have conductors $p^v$ and $N'$, respectively. Let $\Zp^{\otimes k}[\chi]$ be the $p$-adic $\znx$-representation, whose underlying module is $\Zp[\chi]$ and where $(a,b)\in \Zpx\x\zx{N'}$ acts on $\Zp[\chi]$ by multiplication by $a^k\cdot \chi_p(a)\cdot\chi'(b)$. 
	\begin{thm}[Main Theorem]\label{thm:main}
		Let $\Ical\trianglelefteq\Zp[\chi]$ be an ideal and $k \ge 3$ be an integer. Then the followings are equivalent:
		\begin{enumerate}[label=\textup{(\roman*).}]
			\item There is a modular  form $f\in H^0(\Mello(p^v,\Gamma_1(N')),\bfo{k}\otimes\Zp[\chi])^{\chi^{-1}}$ such that $f(q)\in 1+\Ical q\Zp[\chi] \llb q\rrb$.
			\item There is a generator $\gamma\in H^0(\Mello(\Gamma_0(N')),\bfomega^{k,\chi})$ such that $F^{k,\chi}(\gamma)\equiv \gamma\mod \Ical$. 
			\item $\wh{C}^{k,\chi}[\Ical]\cong  (\Gmh\otimes \Zp[\chi])[\Ical]$.
			\item The Galois descent data $\rho^{k,\chi}\colon \piet(\Mello(\Gamma_0(N'))\to (\Zp[\chi])^\x$ of $\wh{C}^{k,\chi}$ is trivial modulo $\Ical$.
			\item The character $\Zpx\x\zx{N'}\xrightarrow{(a,b)\mapsto \chi_p(a)\chi'(b)a^k}(\Zp[\chi])^\x$ is trivial modulo $\Ical$.
		\end{enumerate}
	\end{thm}
	\begin{rem}
		When the character $\chi$ is trivial, we recover Katz's algebro-geometric explanation of congruences of $p$-adic Eisenstein series of level $1$ in \cite[Corollary 4.4.1]{padic}. In that case, Step I in the proof above is not needed.
	\end{rem}
	\begin{rem}\label{rem:other_method}
		The implication from (i) to (v) in \Cref{thm:main} also follows from \cite[Lemma 1.2.2]{Katz_Eisenstein_measure} when $p$ does not divide $N$. 
	\end{rem}
	The proof of \Cref{thm:main} has three steps, which will be explained in details in the rest of this section. The meanings of the symbols in the statement of the theorem are explained in the proof sketch below. 
	\begin{enumerate}[label=\Roman*., leftmargin=*]
		\item Viewing the Dirichlet character $\chi$ as a Galois cohomology class, we construct a formal $\Zp[\chi]$-module $\wh{C}^{k,\chi}$ of height $1$ over $\Mello(\Gamma_0(N'))$ such that 
		\begin{equation*}
		H^0(\Mello(p^v,\Gamma_1(N')),\bfo{k}\otimes \Zp[\chi])^{\chi^{-1}}\cong  H^0(\Mello(\Gamma_0(N')),\bfomega(\wh{C}^{k,\chi})).
		\end{equation*} 
		In this way, we translate congruences of modular forms in $H^0(\Mello(p^v,\Gamma_1(N')),\bfo{k}\otimes\Zp[\chi])^{\chi^{-1}}$ to those of elements in the Dieudonn\'e module of $\wh{C}^{k,\chi}$.
		\item We relate the congruence of the Dieudonn\'e module $\Dbb(\wh{C}^{k,\chi})$ with that of the Galois descent data $[\rho^{k,\chi}]$ for $\wh{C}^{k,\chi}$ by reformulating a Riemann-Hilbert type correspondence in \cite{padic} using the Dieudonn\'e theory of height $1$ formal $A$-modules and their finite subgroups
		\item The Galois cohomology class $[\rho^{k,\chi}]\in H_c^1(\piet(\Mello(\Gamma_0(N')));(\Zp[\chi])^\x)$ is represented by a group homomorphism that factors as
		\begin{equation*}
		\rho^{k,\chi}\colon \piet(\Mello(\Gamma_0(N')))\xrightarrow{\rho^1\x \lambda_{N'}} \Zpx\x\zx{N'}\xrightarrow{(a,b)\mapsto \chi_p(a)\chi'(b)a^k} (\Zp[\chi])^\x.
		\end{equation*}
		Here $\rho^1\colon \piet(\Mello(\Gamma_0(N')))\to \Zpx$ is the Galois descent data for $\wh{C}$ described in \Cref{exmp:FG_ell} and $\lambda_{N'}\colon \piet(\Mello(\Gamma_0(N')))\to \zx{N'}$ classifies the $\zx{N'}$-torsor $\Mello(\Gamma_1(N'))\to \Mello(\Gamma_0(N'))$. The theorem then follows from the surjectivity of $\rho^1\x\lambda_{N'}$. 
	\end{enumerate}	
	\subsection{Step I: Dirichlet characters and Galois descent} The first step in the proof of the Main Theorem is to view the Dirichlet character $\chi\colon\znx\to \Cpx$ as the Galois descent data for a formal $A$-module $\wh{C}^{k,\chi}$ of height $1$ over $\Mello(\Gamma_0(N'))$ along the $\znx$-torsor $\xi\colon  \Mello(p^v,\Gamma_1(N'))\longrightarrow \Mello(\Gamma_0(N'))$. 
	\begin{const}\label{con:gal_desc}
		Let $(C,\eta_p,\eta')$ be the universal elliptic curve with the given level structures over $\Mello(p^v,\Gamma_1(N'))$ and $\wh{C}$ be its formal group. Then $\wh{C}^{\otimes k}\otimes  \Zp[\chi]$ is a formal $\Zp[\chi]$-module of height $1$. Notice that:
		\begin{itemize}
			\item The automorphism group of $\wh{C}^{\otimes k}\otimes \Zp[\chi]$ as a formal $\Zp[\chi]$-module is $(\Zp[\chi])^\x$.
			\item The forgetful map $\xi\colon \Mello(p^v,\Gamma_1(N'))\to \Mello(\Gamma_0(N'))$ is a $\zx{p^v}\x \zx{N'}\cong  \znx$-torsor.
		\end{itemize}
		The Dirichlet character $\chi\colon\znx\to \Cpx$ then represents a cohomology class 
		\begin{align*}
		[\chi]\in &H^1(\znx;(\Zp[\chi])^\x)\\\cong & H^1(\aut_{\Mello(\Gamma_0(N'))}(\Mello(p^v,\Gamma_1(N')));\aut_{\text{formal $\Zp[\chi]$-mod}}(\wh{C}^{\otimes k}\otimes \Zp[\chi])),
		\end{align*}
		where $\znx$ acts on $(\Zp[\chi])^\x$ trivially. This cohomology group classifies formal $\Zp[\chi]$-modules $\Gh$ over $\Mello(\Gamma_0(N'))$ such that $\xi^*\Gh\cong  \wh{C}^{\otimes k}\otimes \Zp[\chi]$ over $\Mello(p^v,\Gamma_1(N'))$ up to isomorphisms. In this way, the cohomology class $[\chi]$ corresponds to a formal $\Zp[\chi]$-module $\wh{C}^{k,\chi}$ over $\Mello(\Gamma_0(N'))$. More precisely, fix an isomorphism $\eta\colon \xi^*\wh{C}^{k,\chi}\simto \wh{C}^{\otimes k}\otimes \Zp[\chi]$, then for any $\sigma\in \znx\cong  \aut_{\Mello(\Gamma_0(N'))}(\Mello(p^v,\Gamma_1(N')))$, we have a commutative diagram of isomorphisms:
		\begin{equation*}
		\begin{tikzcd}
		\xi^*\wh{C}^{k,\chi}\dar["\eta"']\rar["\sigma\otimes 1"]& \sigma^*\xi^*\wh{C}^{k,\chi}\dar["\sigma^*\eta"']\rar[equal]&\xi^*\wh{C}^{k,\chi}\dar["\sigma^*\eta"']\\
		\wh{C}^{\otimes k}\otimes \Zp[\chi]\rar\ar[rr,"{[\chi(\sigma)]}",bend right=20,start anchor=south east, end anchor= south west]&\sigma^*(\wh{C}^{\otimes k}\otimes \Zp[\chi])\rar[equal]&\wh{C}^{\otimes k}\otimes \Zp[\chi]
		\end{tikzcd}
		\end{equation*}
		In this diagram, 
		\begin{itemize}
			\item The homomorphism $[\chi(\sigma)]$ is defined in \Cref{con:Formal_A-mod}.
			\item The isomorphism $\sigma^*\eta$ is the same as $\eta$ since $\znx$ acts on $(\Zp[\chi])^{\x}$ trivially.
			\item The correspondence between $\wh{C}^{k,\chi}$ and $\chi$ is independent of the choice of the isomorphism $\eta$, since the group $\aut_{\Zp[\chi]}(\wh{C}^{\otimes k}\otimes \Zp[\chi])=(\Zp[\chi])^\x$ is abelian.
		\end{itemize}		
	\end{const}
	Let $\bfomega^{k,\chi}=\bfomega(\wh{C}^{k,\chi})$ be the sheaf of invariant differentials of $\wh{C}^{k,\chi}$. The sheaf $\bfomega^{k,\chi}$ is  locally free finitely generated over $\Mello(\Gamma_0(N'))$, since it is the cotangent sheaf of a formal scheme that is \'etale locally isomorphic to $\wh{\mathbb{A}}^r$, where $r$ is the rank of $\Zp[\chi]$ as a $\Zp$-module.
	\begin{prop}\label{prop:omega_k_chi_cohomology}
		We have an isomorphism of locally free sheaves $\xi^*\bfomega^{k,\chi}\cong  \bfo{k}\otimes \Zp[\chi]$ over the stack $\Mello(p^v,\Gamma_1(N'))$. The sheaf cohomology of $\bfomega^{k,\chi}$ is computed as follows: 
		\begin{enumerate}
			\item For all integers $N>1$, we have
			\begin{equation}\label{eqn:omega_k_chi_descent}
				H^0(\Mello(p^v,\Gamma_1(N')),\bfo{k}\otimes\Zp[\chi])^{\chi^{-1}}\cong  H^0(\Mello(\Gamma_0(N')),\bfomega^{k,\chi}).
			\end{equation} 
			\item When $N>3$ or $N=3$ and $p\not\equiv 1\mod 3$, we have for all $s\ge 0$:
			\begin{equation*}
			H^s(\Mello(\Gamma_0(N')),\bfomega^{k,\chi})\cong  H^s(\znx;H^0(\Mello(p^v,\Gamma_1(N')),\bfo{k}\otimes\Zp[\chi])).
			\end{equation*}
			\item When $p\nmid\phi(N)=|\znx|$, we have for all $t\ge 0$:
			\begin{equation*}
			H^t(\Mello(\Gamma_0(N')),\bfomega^{k,\chi})\cong  H^t(\Mello(p^v,\Gamma_1(N')),\bfo{k}\otimes\Zp[\chi])^{\chi^{-1}}.
			\end{equation*}
			\item In particular, when $N$ and $p$ satisfy the conditions in (2) and (3), we further have:
			\begin{equation*}
			H^s(\Mello(\Gamma_0(N')),\bfomega^{k,\chi})=\left\{\begin{array}{cl}
			H^0(\Mello(p^v,\Gamma_1(N')),\bfo{k}\otimes\Zp[\chi])^{\chi^{-1}},&s=0;\\
			0,&\textup{otherwise.}\\ 
			\end{array}\right.
			\end{equation*}
		\end{enumerate}	
	\end{prop}
	\begin{proof}
		The functor $\bfomega$ is compatible with pullbacks, yielding
		\begin{equation*}
		\xi^*\bfomega^{k,\chi}=\xi^*\bfomega(\wh{C}^{k,\chi})\cong  \bfomega(\xi^*\wh{C}^{k,\chi})\cong  \bfomega(\wh{C}^{\otimes k}\otimes \Zp[\chi])=\bfo{k}\otimes \Zp[\chi].
		\end{equation*}		
		To compute $H^s(\Mello(\Gamma_0(N')),\bfomega^{k,\chi})$, we use the Hochschild-Serre spectral sequence \cite[Theorem 2.20]{Milne_etale_cohomology}:
		\begin{equation}\label{eqn:HSSS}
		E^{s,t}_2=H^s(\znx;H^t(\Mello(p^v,\Gamma_1(N')),\xi^*\bfomega^{k,\chi}))\Longrightarrow H^{s+t}(\Mello(\Gamma_0(N')),\bfomega^{k,\chi}),
		\end{equation}
		where $\sigma\in\znx$ acts on $\xi^*\bfomega^{k,\chi}\cong  \bfo{k}\otimes \Zp[\chi]$ by the Galois descent data $1\otimes \chi(\sigma)$. As the spectral sequence is concentrated in the first quadrant, its $E_2^{0,0}$-term receives or supports no differentials. This yields (1).
		
		By \Cref{prop:mu_3_rigid}, the stack $\Mello(p^v,\Gamma_1(N'))$ is a formal affine scheme when $N\ge 4$ or $N=3$ and $p\not\equiv 1\mod 3$. It follows that \eqref{eqn:HSSS} is concentrated in the $t=0$ line in those cases. As a result, the spectral sequence collapses on the $E_2$-page and we have proved (2).
		
		When $p\nmid\phi(N)=|\znx|$, the group cohomology of $\znx$ with coefficients in $\Zp$-modules vanishes in positive degrees. It follows that \eqref{eqn:HSSS} is concentrated in the $s=0$ line in this case and thus collapses on the $E_2$-page. This implies (3). 
		
		Case (4) is the intersection of cases (2) and (3). 
	\end{proof}
	\begin{rem}
		Note that $2$ is the only prime $p$ dividing $\phi(3)=2$. The spectral sequence \eqref{eqn:HSSS} collapses on the $E_2$-page for all $N\ge 3$ and $p$.
	\end{rem}
	Write $\Dbb(\wh{C}^{k,\chi})=(\bfomega^{k,\chi},F^{k,\chi}\colon \varphi^*\bfomega^{k,\chi}\simto \bfomega^{k,\chi})$. The Frobenius homomorphism $F^{k,\chi}$ of $\wh{C}^{k,\chi}$ descends from that of $\xi^*\wh{C}^{k,\chi}\cong  \wh{C}^{\otimes k}\otimes \Zp[\chi]$. \Cref{exmp:FG_ell} and \Cref{con:Formal_A-mod} yield
	\begin{equation*}
	\xi^*F^{k,\chi}=F^{\otimes k}\otimes 1\colon \varphi^*\bfo{k}\otimes\Zp[\chi]\simto \bfo{k}\otimes\Zp[\chi].
	\end{equation*}
	Notice $F^{\otimes k}\otimes 1$ commutes with the Galois descent data $1\otimes \chi(\sigma)$ for $\sigma\in\znx$, we have shown: 
	\begin{prop}[Step I] \label{prop: Eisenstein_Dieudonne}
		Let $f\in H^0(\Mello(p^v,\Gamma_1(N')),\bfo{k}\otimes\Zp[\chi])^{\chi^{-1}}\cong  H^0(\Mello(\Gamma_0(N')),\bfomega^{k,\chi})$ be a modular form. Then $F^{k,\chi}(f(q))=(F^{\otimes k}\otimes 1)(f(q))=f(q^p)$. Let $\Ical\trianglelefteq\Zp[\chi]$ be an ideal. The followings are equivalent:
		\begin{enumerate}[label=\textup{(\roman*).}]
			\item There is a modular form  $f\in H^0(\Mello(\Gamma_0(N')),\bfomega^{k,\chi})$ such that $f(q)\in 1+\Ical q\llb q\rrb$.
			\item There is a generator $\gamma\in H^0(\Mello(\Gamma_0(N')),\bfomega^{k,\chi})$ as an $H^0(\Mello(\Gamma_0(N')),\mathcal{O})\otimes \Zp[\chi]$-module such that $F^{k,\chi}(\gamma)\equiv \gamma\mod \Ical$.
		\end{enumerate}
	\end{prop}
	\begin{rem}
		There is no guarantee that the modular $f$ above is in the Eisenstein space $\Ecal_k(p^v,\Gamma_1(N'),\chi)$. 
	\end{rem}
	This concludes step I in \Cref{subsec:sketch_proof}.
	\subsection{Step II: From Dieudonn\'e modules to Galois representations}
	One major tool Katz used in \cite[Chapter 4]{padic} to explain the congruences of the normalized Eisenstein series $E_{2k}$ of level $1$ is a Riemann-Hilbert type correspondence. In this subsection, we reformulate the correspondence in terms of formal $A$-modules and their finite subgroup schemes, and then apply it to the formal $\Zp[\chi]$-module $\wh{C}^{k,\chi}$ over $\Mello(\Gamma_0(N'))$ we constructed in \Cref{con:gal_desc}. 
	
	Let $\kappa$ be a perfect field of characteristic $p$ containing $\mathbb{F}_q$ and $\W_m(\mathbb{F}_q)$ be the ring of Witt vectors of length $m$ on $\mathbb{F}_q$.  Let $S_m$ be a flat affine $\W_m(\kappa)$-scheme whose special fiber is normal, reduced, and irreducible. Assume $S_m$ is formally smooth, so that it admits an endomorphism $\varphi\colon  S_m\to S_m$ that lifts the $q$-th power map on $S_m/p$. Then Katz proved
	\begin{thm}\textup{\cite[Proposition 4.1.1, Remark 4.1.2.1]{padic}} \label{thm:padic41}\\
		There is an equivalence of closed symmetric monoidal categories:
		\begin{equation*}\left\{\begin{array}{c}
		\text{Finite locally free sheaves } \F \text{ on }S_m\\
		\text{with an isomorphism } F\colon \varphi^*\F\simto\F
		\end{array}\right\} \simeq \left\{\begin{array}{c}
		\text{Finite free }\W_m(\mathbb{F}_q)\text{-modules}\\
		\text{ with continuous }\piet(S_m)\text{-actions}
		\end{array}\right\}.\end{equation*}
	\end{thm}
	\begin{prop}\textup{\cite[Remark 5.5]{Katz_Dwork}}\label{thm:padic41Wk}
		\Cref{thm:padic41} holds for affine formal schemes $S$ over $\W(\kappa)$ under the same assumption. That is, there is an equivalence of closed symmetric monoidal categories:
		\begin{equation*}\left\{\begin{array}{c}
		\text{Finite locally free sheaves } \F \text{ on }S\\
		\text{with an isomorphism } F\colon \varphi^*\F\simto\F
		\end{array}\right\} \simeq \left\{\begin{array}{c}
		\text{Finite free }\W(\mathbb{F}_q)\text{-modules}\\
		\text{ with continuous }\piet(S)\text{-actions}
		\end{array}\right\}.\end{equation*}
	\end{prop}
	This equivalence of Katz is essentially an equivalence of Dieudonn\'{e} module and Galois descent data of a formal group and its finite subgroups. Let $A$ be a $\Zp$-algebra that is finite free as a $\Zp$-module and $\Gh$ be formal $A$-module of height $1$. Let $\Ical\trianglelefteq A$ be an ideal. 
	\begin{defn}
		Define $\Gh[\Ical]$ to be the kernel of all the endomorphisms in $\Ical\trianglelefteq A\hookrightarrow \mathrm{End}(\Gh)$. If $\Gh=\spf R\llb t\rrb$ has a coordinate, then $\Gh[\Ical]=\spf R\llb t \rrb\left/([a](t)\mid a\in \Ical)\right.$ as a finite flat scheme. When $\Ical=(a)$ is a principal ideal, $\Gh[\Ical]=\Gh[a]=\spf R\llb t\rrb/([a](t))$.
	\end{defn}
	\begin{prop}\label{prop:fin_gp_d-mod}
		Let $\Gh$ be a formal $A$-module. Write the Dieudonn\'e module of $\Gh$ as $\Dbb(\Gh)=(M,F,V)$. Then $M$ has an $A$-module structure and the homomorphisms $F$ and $V$ are $A$-linear. The Dieudonn\'{e} module of $\Gh[\Ical]$ is $\Dbb(\Gh)/\Ical=(M/\Ical M,F\colon  \varphi^*(M/\Ical M)\to M/\Ical M,V\colon M/\Ical M\to\varphi^*(M/\Ical M))$.
	\end{prop}
	\begin{prop}\label{prop:fin_gp_gal_desc}
		Let $\Gh$ be a formal $A$-module over $R$ that is isomorphic to $\Gh'$ over the separable closure $R^{\sep}$ of $R$. Let the cohomology class $[\rho]\in H_c^1(\piet(R);\aut_{A}(\Gh'))$ be the Galois descent data for $\Gh$. $[\rho]$ is represented by some crossed homomorphism $\rho\colon  \piet(R)\to\aut_{A}(\Gh')$. Then the Galois descent data for the finite flat group scheme $\Gh[\Ical]$ is represented by the crossed homomorphism:
		\begin{equation*}
		\rho_{\Ical}\colon  \piet(R)\xrightarrow{\rho}\aut_{A}(\Gh')\longrightarrow \aut(\Gh'[\Ical]),
		\end{equation*}
		where the last map $\aut_{A}(\Gh')\longrightarrow \aut(\Gh'[\Ical])$ is the restriction of the quotient map to units \begin{equation*}
		\begin{tikzcd}
		\mathrm{End}_A(\Gh')\rar[twoheadrightarrow]& \mathrm{End}_A(\Gh')/(\Ical\otimes_A \mathrm{End}_A(\Gh')) \cong  \mathrm{End}_A(\Gh'[\Ical])
		\end{tikzcd}
		\end{equation*} 
	\end{prop}
	In the view of \Cref{prop:fin_gp_d-mod} and \Cref{prop:fin_gp_gal_desc}, Katz's Riemann-Hilbert correspondence (\Cref{thm:padic41}) can be generalized as:
	\begin{thm}\label{thm:rhc_formal_A-mod}
		Let $\Gh$ be a formal $A$-module of height $1$ over $R$, where $\spf R$ satisfies the same assumptions as in \Cref{thm:padic41}. Let $\Dbb(\Gh)=(M,F\colon \varphi^*M\simto M)$ and $\rho\colon \piet(R)\to A^\x$ be the Dieudonn\'e module and Galois descent data for $\Gh$, respectively. Then the followings are equivalent:
		\begin{enumerate}
			\item There is a generator $\gamma$ of $M$ as an $R\otimes A$-module such that $F\gamma\equiv \gamma\mod \Ical$.
			\item $\Gh[\Ical]\cong  (\Gmh\otimes A)[\Ical]$.
			\item The composition homomorphism $\rho_{\Ical}\colon \piet(R)\xrightarrow{\rho}A^\x\twoheadrightarrow (A/\Ical)^\x$ is trivial.
		\end{enumerate} 
	\end{thm}
	\begin{proof}
		Let's prove the case when $R=R/p$. By \cite[Main Theorem 1]{de_Jong_crystalline_Dieudonne}, the functor $\Dbb$ is an equivalence over $R$. The claim then follows from the computation of the Dieudonn\'e module and the Galois descent data of $\Gmh$ in \Cref{exmp:D_Gmh}, as well as \Cref{prop:fin_gp_d-mod} and \Cref{prop:fin_gp_gal_desc}. 
		
		Now let $R$ be a $\W\kappa$-algebra. Using the Lubin-Tate deformation theory, we can show there is an equivalence between height $1$ formal groups over $R/p$ and their deformations to $R/p$. The claim now follows from the $R=R/p$-case.
	\end{proof}
	\begin{rem}
		Katz's \Cref{thm:padic41} is the $\Ical=(p^m)\trianglelefteq A=\W\mathbb{F}_{q}$ case of \Cref{thm:rhc_formal_A-mod}. 
	\end{rem}
	\begin{rem}
		We can generalize \Cref{thm:padic41} and \Cref{thm:padic41Wk} in terms of formal groups and formal $A$-modules of height/slope $h>1$. In that case, we need to study the Dieudonn\'e module of the Honda formal group $\Gamma_h$ of height $h$ and its finite subgroup schemes. The result is included in \Cref{Sec:appendix}.
	\end{rem}
	Now apply \Cref{thm:rhc_formal_A-mod} to the formal $\Zp[\chi]$-module $\wh{C}^{k,\chi}$ over $\Mello(\Gamma_0(N'))$ constructed in \Cref{con:gal_desc}, we have established Step II in \Cref{subsec:sketch_proof}:
	\begin{cor}[Step II] \label{cor:rhc_mello}
		Let $\Ical\trianglelefteq \Zp[\chi]$ be an ideal. The followings are equivalent:
		\begin{enumerate}[start=2,label=\textup{(\roman*).}]
			\item There is a generator $\gamma\in H^0(\Mello(\Gamma_0(N')),\bfomega^{k,\chi})$ such that $F^{k,\chi}(\gamma)\equiv \gamma\mod \Ical$.
			\item $\wh{C}^{k,\chi}[\Ical]\cong  (\Gmh\otimes \Zp[\chi])[\Ical]$.
			\item The Galois descent data $\rho^{k,\chi}\colon \piet(\Mello(\Gamma_0(N'))\to (\Zp[\chi])^\x$ of $\wh{C}^{k,\chi}$ is trivial modulo $\Ical$.
		\end{enumerate}
	\end{cor}
	\subsection{Step III: Factorizations of the Galois descent data}
	The final step is to study the Galois descent data $\rho^{k,\chi}$ for $\wh{C}^{k,\chi}$. Recall from \Cref{con:gal_desc}, $\wh{C}^{k,\chi}$ is constructed using the following data:
	\begin{itemize}
		\item $\xi^*\wh{C}^{k,\chi}\cong  \wh{C}^{\otimes k}\otimes \Zp[\chi]$ over $\Mello(p^v,\Gamma_1(N'))$, where $\xi\colon \Mello(p^v,\Gamma_1(N'))\to \Mello(\Gamma_0(N'))$ is the forgetful map.
		\item $\wh{C}^{k,\chi}$ corresponds to descent data $[\chi]\in H^1(\znx;(\Zp[\chi])^\x)$.
	\end{itemize}
	\begin{prop}\label{prop:rho_kchi_factorization}
		$\rho^{k,\chi}\colon \piet(\Mello(\Gamma_0(N'))\to (\Zp[\chi])^\x$ factors as 
		\begin{equation*}
		\rho^{k,\chi}\colon \piet(\Mello(\Gamma_0(N')))\xrightarrow{\rho^1\x \lambda_{\xi}}\Zpx\x \znx\xrightarrow{(-)^k\cdot\chi(-)}(\Zp[\chi])^\x,
		\end{equation*}
		where $\lambda_{\xi}\colon \piet(\Mello(\Gamma_0(N')))\to \znx$ is the character that classifies the $\znx$-torsor $\xi$. 
	\end{prop}
	\begin{proof}
		Recall in \Cref{con:gal_desc}, we used the following correspondence to construct $\wh{C}^{k,\chi}$ from the character $\chi$:
		\begin{equation}\label{eqn:char_desc}
		H^1(\znx;(\Zp[\chi])^\x)\cong  \left.\left\{\begin{array}{c}
		\text{Formal $\Zp[\chi]$-modules $\Gh$ over } \Mello(\Gamma_0(N'))\\
		\text{such that }\xi^*\Gh\cong  \wh{C}^{\otimes k}\otimes \Zp[\chi]\text{ over }\Mello(p^v,\Gamma_1(N'))
		\end{array}\right\}\right/\sim.
		\end{equation}
		Here, the constant group homomorphism on the left hand side corresponds to the formal $\Zp[\chi]$-module $ \wh{C}^{\otimes k}\otimes \Zp[\chi]$ over $\Mello(\Gamma_0(N'))$. Now we need to describe this correspondence in terms of the Galois descent data $\rho_{\Gh}$ of $\Gh$. On one hand, since $\xi^*\Gh\cong  \wh{C}^{\otimes k}\otimes \Zp[\chi]$, the composition
		\begin{equation}\label{eqn:gal_desc_pullback}
		\piet(\Mello(p^v,\Gamma_1(N')))\xrightarrow{\piet(\xi)}\piet(\Mello(\Gamma_0(N')))\xrightarrow{\rho_{\Gh}}(\Zp[\chi])^\x
		\end{equation}
		is the same as the Galois descent data for the formal $\Zp[\chi]$-module $\wh{C}^{\otimes k}\otimes \Zp[\chi]$ over $\Mello(p^v,\Gamma_1(N'))$. On the other hand, by \Cref{exmp:FG_ell} and \Cref{con:Formal_A-mod}, this Galois descent data also factors as
		\begin{equation}\label{eqn:gal_desc_con}
		\piet(\Mello(p^v,\Gamma_1(N')))\xrightarrow{\piet(\xi)}\piet(\Mello(\Gamma_0(N')))\xrightarrow{\rho^1}\Zpx\xrightarrow{(-)^k}\Zpx\xrightarrow{i} (\Zp[\chi])^\x.
		\end{equation}
		Denote the composition $i\circ(-)^k\circ \rho^1$ in \eqref{eqn:gal_desc_con} by $\rho^k$. Since the first maps in \eqref{eqn:gal_desc_pullback} and \eqref{eqn:gal_desc_con} are both $\piet(\xi)$ and the compositions are the same, the difference of $\rho_{\Gh}$ and $\rho^k$ must factor through the cokernel of $\piet(\xi)$. We have the following diagram:
		\begin{equation*}
		\begin{tikzcd}
		\piet(\Mello(p^v,\Gamma_1(N')))\rar["\piet(\xi)"]&\piet(\Mello(\Gamma_0(N')))\rar["\lambda_\xi"]\dar[shift right,"\rho_{\Gh}"']\dar[shift left,"\rho^k"]&\znx\rar\ar[dl,dashed,"\exists!~\chi_{\Gh}"] &1\\
		&(\Zp[\chi])^\x&&
		\end{tikzcd}
		\end{equation*}
		As the cokernel of $\piet(\xi)$, $\lambda_{\xi}$ classifies the $\znx$-torsor $\xi\colon \Mello(p^v,\Gamma_1(N'))\to \Mello(\Gamma_0(N'))$. It follows the that there exists a unique character $\chi_{\Gh}\colon \znx\to\Zp[\chi]$ such that for any $\sigma\in\piet(\Mello(\Gamma_0(N')))$, $\rho_{\Gh}(\sigma)=(\rho^1(\sigma))^k\cdot (\chi_{\Gh}\circ\lambda_\xi)(\sigma)$.
		
		This $\chi_{\Gh}$ is the character corresponding to $\Gh$ in \eqref{eqn:char_desc}. Since $\wh{C}^{k,\chi}$ is constructed using $\chi$, we have 
		\begin{equation*}
		\rho^{k,\chi}(\sigma)=(\rho^1(\sigma))^k\cdot (\chi\circ\lambda_\xi)(\sigma)=((-)^k\cdot \chi(-))\circ (\rho^1\x \lambda_\xi)(\sigma)
		\end{equation*}
		for all $\sigma\in \piet(\Mello(\Gamma_0(N')))$.
	\end{proof}
	Now we need to find the image of $\rho^1\x \lambda_{\xi}$. 
	\begin{prop}\label{prop:rho_lambda_factorization}
		$\rho^1\x \lambda_{\xi}\colon \piet(\Mello(\Gamma_0(N')))\longrightarrow \Zpx\x \znx$ factors as:
		\begin{equation*}
		\rho^1\x \lambda_{\xi}\colon  \piet(\Mello(\Gamma_0(N')))\xrightarrow{\rho^1\x \lambda_{N'}}\Zpx\x\zx{N'}\xrightarrow{(a,b)\mapsto (a,[a],b)} \Zpx\x\zx{p^v}\x\zx{N'}\cong  \Zpx\x \znx,
		\end{equation*}
		where $\lambda_{N'}\colon  \piet(\Mello(\Gamma_0(N')))\to \zx{N'}$ classifies the $\zx{N'}$-torsor $\Mello(\Gamma_1(N'))\to \Mello(\Gamma_0(N'))$.
	\end{prop}
	\begin{proof}
		We prove the factorization by translating Galois representations into torsors over $\Mello(\Gamma_0(N'))$. 
		\begin{lem}\label{lem:rho_1_torsor}
			The character $\rho^1\colon \piet(\Mello(\Gamma_0(N')))\to\Zpx$ classifies the $\Zpx$-torsor $\Mtriv(\Gamma_0(N'))\to \Mello(\Gamma_0(N'))$, where $\Mtriv(\Gamma_0(N'))$ is a stack whose $R$-points are 
			\begin{equation*}
			\Mtriv(\Gamma_0(N'))(R)=\{(C/R,\eta\colon\Gmh\simto\wh{C},H\subseteq C[N'])\mid H\cong  \underline{\Z/N'} \}.
			\end{equation*}
		\end{lem}
		\begin{proof}[Proof of the Lemma]
			Recall that $[\rho^1]\in H_c^1(\piet(\Mello(\Gamma_0(N')));\Zpx)$ is the Galois descent data for $\wh{C}$, the formal group of the universal elliptic curve over $\Mello(\Gamma_0(N'))$. The character $\rho^1$ then corresponds to a $\Zpx$-torsor over $\Mello(\Gamma_0(N'))$ such that its fiber over the an $R$-point $(C/R,H\subseteq C[N'])$ is the set of triples $(C/R,\eta\colon\Gmh\simto \wh{C},H\subseteq C[N'])$. 
		\end{proof}
		\Cref{lem:rho_1_torsor} implies that the character $\rho^1\x \lambda_{\xi}$ classifies the $\Zpx\x \znx$-torsor $\Mtriv(p^v,\Gamma_1(N'))\to \Mello(\Gamma_0(N'))$, where $\Mtriv(p^v,\Gamma_1(N'))$ is a stack whose $R$-points are
		\begin{equation*}
		\Mtriv(p^v,\Gamma_1(N'))(R)=\{(C/R,\eta,\eta_p,\eta')\mid \eta\colon \Gmh\simto\wh{C},\eta_p\colon \mu_{p^v}\simto\wh{C}[p^v],\eta'\colon \underline{\Z/N'}\hookrightarrow C[N']\}.
		\end{equation*}
		Sitting in between $\Mtriv(p^v,\Gamma_1(N'))$ and $\Mello(\Gamma_0(N'))$ is the stack $\Mtriv(\Gamma_1(N'))$, whose $R$-points are 
		\begin{equation*}
		\Mtriv(\Gamma_1(N'))(R)=\{(C/R,\eta,\eta')\mid \eta\colon \Gmh\simto\wh{C},\eta'\colon \underline{\Z/N'}\hookrightarrow C[N']\}.
		\end{equation*}
		In the $\Zpx\x\znx$-torsor
		\begin{equation*}
		\begin{tikzcd}
		\Mtriv(p^v,\Gamma_1(N'))\rar&\Mtriv(\Gamma_1(N'))\rar&\Mello(\Gamma_0(N')),
		\end{tikzcd}		
		\end{equation*}
		the first map $\Mtriv(p^v,\Gamma_1(N'))\longrightarrow \Mtriv(\Gamma_1(N'))$ is a $\zx{p^v}$-torsor that admits a section:
		\begin{equation*}
		s\colon \Mtriv(\Gamma_1(N')) \longrightarrow \Mtriv(p^v,\Gamma_1(N')), \quad (C/R,\eta,\eta')\longmapsto (C/R,\eta,\eta|_{\wh{C}[p^v]},\eta').
		\end{equation*}
		The existence of this section implies that $\rho^1\x \lambda_\xi$ must factor through $\rho^1\x \lambda_{N'}\colon \piet(\Mello(\Gamma_0(N')))\to \Zpx\x\zx{N'}$, the character corresponding to the $\Zpx\x\zx{N'}$-torsor $\Mtriv(\Gamma_1(N'))\to \Mello(\Gamma_0(N'))$. The formula of $s$ then yields a commutative diagram:
		\begin{equation*}
		\begin{tikzcd}[anchor=south,column sep=4 em]
		\piet(\Mello(\Gamma_0(N')))\rar["\rho^1\x \lambda_{\xi}"]\dar["\rho^1\x \lambda_{N'}"']&\Zpx\x\znx\dar[equal]\\
		\Zpx\x\zx{N'}\rar["{(a,b)\mapsto (a,[a],b)}"]&\Zpx\x\zx{p^v}\x\zx{N'}.
		\end{tikzcd}\qedhere
		\end{equation*}
	\end{proof}
	Combining \Cref{prop:rho_kchi_factorization} and \Cref{prop:rho_lambda_factorization}, we have shown
	\begin{cor}\label{cor:rho_kchi_factorization}
		$\rho^{k,\chi}\colon \piet(\Mello(\Gamma_0(N'))\to (\Zp[\chi])^\x$ factors as
		\begin{equation}\label{eqn:rho_kchi}
		\rho^{k,\chi}\colon \piet(\Mello(\Gamma_0(N')))\xrightarrow{\rho^1\x \lambda_{N'}} \Zpx\x\zx{N'}\xrightarrow{(a,b)\mapsto \chi_p(a)\chi'(b)a^k} (\Zp[\chi])^\x.
		\end{equation}
	\end{cor}
	To relate the congruence of $\rho^{k,\chi}$ with that of the second map in \eqref{eqn:rho_kchi}, it remains to show:
	\begin{prop} \label{prop:char_surj}
		$\rho^1\x \lambda_{N'}\colon \piet(\Mello(\Gamma_0(N')))\longrightarrow \Zpx\x\zx{N'}$ is surjective.
	\end{prop}
	\begin{proof}
		By \cite[Theorem 5.4.2]{Szamuely_gal_gp}, the surjectivity of $\rho^1\x \lambda_{N'}$ is equivalent to the connectivity of the $\Zpx\x\zx{N'}$-torsor it classifies. As $\rho^1\x \lambda_{N'}$ classifies the torsor $\Mtriv(\Gamma_1(N'))\to \Mello(\Gamma_0(N'))$, we need to show $\Mtriv(\Gamma_1(N'))$ is connected.
		
		By a relative version of Igusa's theorem in \cite[Corrollary 12.6.2.(2)]{amec}, $\Mtriv(\Gamma_1(N'))$ is connected whenever $\Mello(\Gamma_1(N'))$ is. The integral stack $\Mell(\Gamma_1(N'))$ has geometrically connected fiber by \cite[Theorem 1.2.1]{Conrad_arithmetic_moduli}. It is also smooth by \cite[Corollary 4.7.1]{amec}. It follows that $\Mell(\Gamma_1(N'))$ is irreducible and so is its $p$-completion $\Mell(\Gamma_1(N'))^\wedge_{p}$. From this we conclude $\Mello(\Gamma_1(N'))$ is irreducible (hence connected), since it is an open substack of an irreducible stack.
	\end{proof}
	Now by \Cref{cor:rho_kchi_factorization} and \Cref{prop:char_surj}, we have proved:
	\begin{cor}[Step III]\label{cor:congruence_rho_kchi}
		Let $\Ical\trianglelefteq \Zp[\chi]$ be an ideal. The followings are equivalent:
		\begin{enumerate}[start= 4,label=\textup{(\roman*).}]
			\item The composition $\rho^{k,\chi}\colon  \piet(\Mello(\Gamma_0(N')))\longrightarrow (\Zp[\chi])^\x\twoheadrightarrow (\Zp[\chi]/\Ical)^\x$ is trivial.
			\item The composition $\Zpx\x\zx{N'}\xrightarrow{(a,b)\mapsto \chi_p(a)\chi'(b)a^k}(\Zp[\chi])^\x\twoheadrightarrow (\Zp[\chi]/\Ical)^\x$ is trivial.
		\end{enumerate}
	\end{cor}
	\section{The maximal congruence of Eisenstein series}
	\Cref{thm:main} identifies congruences of modular forms in  $H^0(\Mello(p^v,\Gamma_1(N')),\bfo{k}\otimes\Zp[\chi])^{\chi^{-1}}$ with that of $\Zp^{\otimes k}[\chi]$ as a $\Zpx\x \zx{N'}$-representation in $\Zp[\chi]$-modules. In this section, we first compute the maximal congruence of $\Zp^{\otimes k}[\chi]$ and then find explicit examples of modular forms that realize this congruence in many cases. In all the examples we study, there is a modular form in the Eisenstein space $\Ecal_k(p^v,\Gamma_1(N'),\chi)$ that realizes the maximal congruence. 
	\subsection{Congruences of $p$-adic representations}
	\begin{defn}\label{defn:max_cong}
		Let $R$ be a $p$-complete local ring and $M$ be a torsion-free $R$-module with a continuous $R$-module action by a profinite group $G$.  $M$ is said to be a trivial $G$-representation modulo an ideal $\Ical\trianglelefteq R$ if $G$ acts on $M/\Ical M$ trivially, or equivalently $(M/\Ical M)^G=M/\Ical M$. The maximal congruence of $M$ as a $G$-representation is the smallest ideal $\Ical$ such that $M/\Ical M$ is a trivial $G$-representation.
	\end{defn}
	\begin{rem}
		The $G$-action on the quotient $M/\Ical M$ is well defined since $G$ acts by $R$-linear maps. Otherwise, we need to assume $\Ical\trianglelefteq R$ is a $G$-invariant ideal, i.e. $g\Ical=\Ical$ for all $g\in G$.
	\end{rem}
	\begin{lem}
		When the underlying $R$-module of the $G$-representation $M$ is $R$, the $G$-action of $M$ is then associated to a character $\chi\colon G\to R^\x$. Let $\{g_i\mid i\in I\}$ be a set of generators of $G$. The maximal congruence of $M$ is the ideal $(1-\chi(g_i)\mid i\in I)$.
	\end{lem}
	\begin{proof}
		The maximal congruence of $M$ is by definition the ideal $(1-\chi(g)\mid g\in G)$. Notice that
		\begin{equation*}
		(1-\chi(gg'))=(1-\chi(g)+\chi(g)-\chi(gg'))\subseteq (1-\chi(g))+(\chi(g)-\chi(gg'))=(1-\chi(g))+(1-\chi(g')).
		\end{equation*}
		and that $(1-\chi(g^{-1}))=(\chi(g)-1)$, we have $(1-\chi(g)\mid g\in G)=(1-\chi(g_i)\mid i\in I)$.
	\end{proof}
	When $p>2$, $\Zpx$ is topologically cyclic. When $p=2$, $\Z_2^\x=\{\pm 1\}\x (1+4\Z_2)$ and $1+4\Z_2$ is topologically cyclic. Let $g$ be a topological generator of $\Zpx$ when $p>2$ and a topological generator of $1+4\Z_2$ when $p=2$.
	\begin{thm}\label{thm:rep_congruence}
		The congruences of $\Zp^{\otimes k}[\chi]$ have seven cases:	
		\begin{enumerate}[label=\textup{\Roman*.},align=left]
			\item $p>2$ and the conductor of $\chi$ is $p$ or $1$. In this case, $\chi=\omega^a$ for some integer $0\le a\le p-2$, where $\omega\colon \zpx\to \Zpx$ is the \textbf{$p$-adic Teichm\"uller character}. The image of $\chi$ is contained in $\Zpx$. Then the maximal congruence of $\Zp^{\otimes k}[\omega^a]$ is the following ideal in $\Zp=\Zp[\omega^a]$:
			\begin{equation*}
			(1-g^k\chi(g))=(1-g^k\omega^a(g))=\left\{\begin{array}{cl}
			(p^{v_p(k)+1}), & (p-1)\mid (k+a);\\
			(1)& \textup{otherwise.}
			\end{array}\right.
			\end{equation*}
			\item $p=2$ and the conductor of $\chi$ is $4$ or $1$. In this case, $\chi=\omega^a$ for $a=0$ or $1$, where $\omega\colon \zx{4}\to \Z_2^\x$ is the $2$-adic Teichm\"{u}ller character. As $g\in 1+4\Z_2$, $\omega(g)=1$. Again the image of $\chi$ is contained in $\Z_2^\x$. Then the maximal congruence of $\Z_2^{\otimes k}[\omega^a]$ is the following ideal in $\Z_2=\Z_2[\omega^a]$:
			\begin{equation*}
			(1-g^k\omega^a(g),1-(-1)^k\omega^a(-1))=\left\{\begin{array}{cl}
			(2^{v_p(k)+2}), & 2\mid (k+a);\\ 
			(2), &\textup{otherwise.}
			\end{array}\right.
			\end{equation*}
			\item $p>2$ and the conductor of $\chi$ is $p^v>p$. In this case, $\zx{p^v}\cong  \zpx\x C_{p^{v-1}}$ and  As $\chi$ is primitive of conductor $p^v$, $\chi|_{C_{p^{v-1}}}$ is injective. As a result, $\Zp[\chi]=\Zp[\zeta_{p^{v-1}}]$. $\Zp[\zeta_{p^{v-1}}]$ is a $p$-complete local ring with uniformizer $1-\zeta_{p^{v-1}}$. Write $\chi|_{\zpx}=\omega^{a}$ for some $0\le a\le p-2$. Then the maximal congruence of $\Zp^{\otimes k}[\omega^a]$ is the following ideal in $\Zp[\zeta_{p^{v-1}}]=\Zp[\chi]$:
			\begin{equation*}
			(1-g^k\chi(g))=(1-\zeta_{p^{v-1}}g^k\omega^a(g))=\left\{\begin{array}{cl}
			(1-\zeta_{p^{v-1}}), &(p-1)\mid (k+a);\\
			(1),& \textup{otherwise.}
			\end{array}\right.
			\end{equation*}
			\item $p=2$ and the conductor of $\chi$ is $2^v>4$. In this case, $\zx{2^v}\cong  \zx{4}\x C_{2^{v-2}}$. As $\chi$ is primitive of conductor $2^v$, $\chi|_{C_{2^{v-2}}}$ is injective. As a result, $\Z_2[\chi]=\Z_2[\zeta_{2^{v-2}}]$. $\Z_2[\zeta_{2^{v-2}}]$ is a $2$-complete local ring with uniformizer $1-\zeta_{2^{v-2}}$. Write $\chi|_{\zx{4}}=\omega^{a}$ for $a=0$ or $1$. Then the maximal congruence of $\Z_2^{\otimes k}[\chi]$ is the following ideal in $\Z_2[\zeta_{2^{v-2}}]=\Z_2[\chi]$:
			\begin{equation*}
			(1-\zeta_{2^{v-2}} g^k\omega^a(g),1-(-1)^k\omega^a(-1))=(1-\zeta_{2^{v-2}})\quad \text{for all $k$ and $a$}.
			\end{equation*}
			\item $N'\neq 1$ and $|\imag \chi'|$ is not a power of $p$. In this case, $\imag\chi'$ contains of a root of unity $\zeta_{n'}$ whose order $n'$ is coprime to $p$. As $1-\zeta_{n'}$ is invertible in $\Zp[\zeta_{n'}]\subseteq \Zp[\chi]$, we have the maximal congruence of $\Zp^{\otimes k}[\chi]$ is the ideal $(1)$ in $\Zp[\chi]$.
			
			\item $p>2$, $N'\neq 1$ and $|\imag \chi'|>1$ is a power of $p$. In the case, $\imag \chi'$ is generated by $\zeta_{p^{v'}}$ for some $v'\ge 1$. We have $\Zp^{\otimes k}[\chi]=\Zp[\zeta_{p^{\max(v-1,v')}}]$. Write $\chi_p|_{\zpx}=\omega^a$ for some $0\le a\le p-2$. Then the maximal congruence of $\Zp^{\otimes k}[\chi]$ is the following ideal in $\Zp^{\otimes k}[\chi]=\Zp[\zeta_{p^{\max(v-1,v')}}]$:
			\begin{equation*}
			(1-g^k\chi(g),1-\zeta_{p^{v'}})=(1-\zeta_{p^{v-1}}g^k\omega^a(g),1-\zeta_{p^{v'}})=\left\{\begin{array}{cl}
			(1-\zeta_{p^{\max(v-1,v')}}), &(p-1)\mid (k+a);\\
			(1),& \textup{otherwise.}
			\end{array}\right.
			\end{equation*}
			
			\item $p=2$, $N'\neq 1$ and $\Q_2(\chi')$ is a totally ramified extension of $\Q_2$. In the case, the image of $\chi'$ is generated by $\zeta_{2^{v'}}$ for some $v'\ge 1$. We have $\Z_2^{\otimes k}[\chi]=\Z_2[\zeta_{2^{\max(v',v-2)}}]$. Write $\chi_2|_{\zx{4}}=\omega^a$ for $a=0, 1$.  Then the maximal congruence of $\Z_2^{\otimes k}[\chi]$ is the following ideal in $\Z_2[\zeta_{2^{\max(v',v-2)}}]=\Z_2[\chi]$:
			\begin{equation*}
			(1-\zeta_{2^{v'}}, 1-\zeta_{2^{v-2}}g^k\omega^a(g),1-(-1)^k\omega^a(-1))=(1-\zeta_{2^{\max(v',v-2)}})\quad \text{for all $k$ and $a$}.
			\end{equation*}
		\end{enumerate}
	\end{thm}
	\subsection{Realizations of the maximal congruence}
	Having computed the maximal congruence of the $\Zpx\x\zx{N'}$-representation $\Zp^{\otimes}[\chi]$, now we give explicit examples of modular forms in the Eisenstein subspace $\Ecal_{k}(p^v,\Gamma_1(N'),\chi)$ whose $q$-expansions realize the maximal congruence.
	
	Let $k\ge 3$ be an integer such that $(-1)^k=\chi(-1)$. Recall from \Cref{thm:eisenstein_subspace_basis} and \Cref{cor:p-adic_basis} that Eisenstein subspace $\Ecal_{k}(p^v,\Gamma_1(N'),\chi)\otimes \Qp$ is spanned by Eisenstein series of the forms:
	\begin{align*}
	E_{k,\chi^0,\chi}(q^t)=c\cdot E_{k,\chi}(q^t)&=~c\cdot\left(1-\frac{2k}{B_{k,\chi}}\sum_{n\ge 1}\left(\sum_{0<d\mid n}\chi(d)d^{k-1} \right)q^{nt}\right),\\
	E_{k,\chi_1,\chi_2}(q^t)&=\sum_{n\ge 1}\left(\sum_{0<d\mid n}\chi_1^{-1}(n/d)\chi_2(d)d^{k-1}\right)q^{nt},
	\end{align*}
	where \begin{itemize}
		\item $c$ is $\Zp[\chi]$ with the smallest valuation so that $E_{k,\chi^0,\chi}(q^t)\in \Zp[\chi]\llb q\rrb$.
		\item $\chi_1$ and $\chi_2$ are characters of conductors $N_1$ and $N_2$ satisfying $\chi_1/\chi_2=\chi^{-1}$ and $(N_1N_2t)\mid N$.
	\end{itemize} 
	By the $q$-expansion principle \Cref{prop:q-exp}, an element of $\Ecal_{k}(p^v,\Gamma_1(N'),\chi)$ is a $\Qp$-linear combination $f(q)$ of these $E_{k,\chi_1,\chi_2}(q)$ such that $f(q)\in\Zp[\chi]\llb q\rrb$. Write $E_{k,\chi^0,\chi}$  and $E_{k,\chi_1,\chi_2}$ for $E_{k,\chi^0,\chi}(q)$  and $E_{k,\chi_1,\chi_2}(q)$, respectively. Using the arithmetic properties of generalized Bernoulli numbers in \cite[Theorems 1 and 3]{Carlitz_GBN}, we can check
	\begin{prop}\label{prop:E_kchi_max_congruence}
		In Cases I-V in \Cref{thm:rep_congruence}, the Eisenstein series $E_{k,\chi}$ realizes the maximal congruence predicted in \Cref{thm:main}.
	\end{prop}
	By \cite[Theorem 1]{Carlitz_GBN}, $\frac{B_{k,\chi}}{k}$ is an algebraic $p$-adic integer when $N$ is not a power of $p$. As a result $E_{k,\chi}(q)$ does not realize the maximal congruence in Cases VI and VII in \Cref{thm:rep_congruence}. Instead, we can consider linear combinations of basis in the Eisenstein subspace. In general, it is hard to write down the explicit formulas of modular forms that satisfies the maximal congruence predicted in \Cref{thm:main} and \Cref{thm:rep_congruence} in Cases VI and VII. Here, we work out one of the simplest cases in the rest of this subsection.
	\begin{exmp}
		Consider the character $\chi\colon\zx{\ell}\to \Cpx$, where $\ell$ is a prime different from $p>2$ and $\Q_p(\chi)$ is a totally ramified extension of $\Qp$.  In this case, $\Zp[\chi]=\Zp[\zeta_{p^m}]$ for some $m\ge 1$ is a $p$-complete local ring with uniformizer $\varpi=1-\zeta_{p^m}$. Write the maximal ideal of $\Zp[\chi]$ by $\mfrak$. By \cite[Theorem 1]{Carlitz_GBN}, $\frac{B_{k,\chi}}{2k}$ is an algebraic $p$-adic integer. As a result, we can take $E_{k,\chi^0,\chi}$ to be:
		\begin{equation*}
		E_{k,\chi^0,\chi}=\frac{B_{k,\chi}}{2k}-\sum_{n\ge 1}\left(\sum_{0<d\mid n}\chi(d)d^{k-1} \right)q^{n}.
		\end{equation*}
		
		Comparing \Cref{thm:main} and Case VI in \Cref{thm:rep_congruence}, we should expect to find a modular form of weight $k$, level $\Gamma_1(\ell)$, and character $\chi$ that is congruent to $1$ modulo $\mfrak=(\varpi)$ only when $(p-1)\mid k$, and there is no modular form of level $\Gamma_1(\ell)$ and character $\chi$ that is congruent to $1$ modulo $\mfrak^2$. The Eisenstein subspace in this case $\Ecal_k(\Gamma_1(\ell),\chi)$ is spanned by $E_{k,\chi^0,\chi}(q)$ and $E_{k,\chi^{-1},\chi^0}(q)$. 
		
		When $(p-1)\mid k$, the maximal congruence is realized by a linear combination of the two basis Eisenstein series with some cusp forms potentially, since neither of them satisfies the maximal congruence relation.  As $\frac{B_{k,\chi}}{2k}\in \Zp[\chi]$, we have $\frac{B_{k,\chi}}{2k}\cdot E_{k,\chi}\in \Zp[\chi]\llb q\rrb$. Notice that the coefficients of $q$ in  $E_{k,\chi^0,\chi}$ and $E_{k,\chi^{-1},\chi^0}$ are $-1$ and $1$, respectively. Consider the $q$-expansion of their sum:
		\begin{equation}\label{eqn:max_cong_qexp}
		E_{k,\chi^0,\chi}+E_{k,\chi^{-1},\chi^0}=\frac{B_{k,\chi}}{2k}+\sum_{n\ge 1}a_nq^n=\frac{B_{k,\chi}}{2k}+\sum_{n\ge 1}\left(\sum_{0<d\mid n}(\chi^{-1}(n/d)-\chi(d))d^{k-1} \right)q^n.
		\end{equation}
		\begin{lem}\label{lem:q-common_factor}
			$E_{k,\chi^0,\chi}+E_{k,\chi^{-1},\chi^0}\equiv \frac{B_{k,\chi}}{2k}\mod \mfrak q\llb q\rrb$ for all $k$ with $(-1)^k=\chi(-1)$.
		\end{lem}
		\begin{proof}
			We need to show the coefficient $a_n$ of $q^n$ in \eqref{eqn:max_cong_qexp} is in $\mfrak$ for all $n\ge 1$. Write $n=\ell^w n'$ where $\ell\nmid n'$. Since the conductor of $\chi$ is the prime number $\ell$, $\chi(a)=0$ iff $\ell\mid a$. As a result, we have
			\begin{align}
			a_n=\sum_{0<d\mid n}(\chi^{-1}(n/d)-\chi(d))d^{k-1}
			&=\sum_{0<d\mid n}\chi^{-1}(n/d)d^{k-1}-\sum_{0<d\mid n}\chi(d)d^{k-1}\nonumber\\
			&=\sum_{\ell^w\mid d\mid n}\chi^{-1}(n/d)d^{k-1}-\sum_{0<d\mid n'}\chi(d)d^{k-1}\nonumber\\
			(\text{set }d=\ell^w d'\text{ in the first summation})\quad &=\sum_{0<d'\mid n'}\chi^{-1}(n'/d')(\ell^wd')^{k-1}-\sum_{0<d\mid n'}\chi(d)d^{k-1}\nonumber\\
			&=\sum_{0<d'\mid n'}\chi^{-1}(n')\chi(d')\ell^{w(k-1)}d'^{k-1}-\sum_{0<d\mid n'}\chi(d)d^{k-1}\nonumber\\
			&=~(\chi^{-1}(n')\ell^{w(k-1)}-1)\sum_{0<d\mid n'}\chi(d)d^{k-1}.\label{eqn:a_n}
			\end{align}
			Since $\zx{\ell}$ surjects onto $C_{p^m}$ by assumption, there is a congruence $\ell\equiv 1\mod p$. This implies $\ell^{w(k-1)}\equiv 1\mod p$. Also, as $\chi^{-1}(n')\neq 0$ is a $p$-power root of unity, we have $1-\chi^{-1}(n')\in \mfrak$. Combining these two facts, we conclude
			\begin{equation*}
			\chi^{-1}(n')\ell^{w(k-1)}-1=\chi^{-1}(n')-1+\chi^{-1}(n')(\ell^{w(k-1)}-1)\in \mfrak
			\end{equation*}
			for all $n'$ not divided by $\ell$. This shows $a_n\in \mfrak$ for all $n$. From this, we conclude  $E_{k,\chi^0,\chi}+E_{k,\chi^{-1},\chi^0}\equiv \frac{B_{k,\chi}}{2k}\mod  \mfrak q\llb q\rrb$. 
		\end{proof}
		\begin{prop}\label{prop:numerator_valuation}
			The algebraic $p$-adic integer $\frac{B_{k,\chi}}{2k}$ is in $\mfrak$ if $(p-1)\nmid k$. 
		\end{prop}
		\begin{proof}
			When $(p-1)\nmid k$, there are no modular forms in $H^0(\Mello(p^v,\Gamma_1(N')),\bfo{k}\otimes\Zp[\chi])^{\chi^{-1}}$  whose $q$-expansion is in $1+\mfrak q\llb q\rrb$ by \Cref{thm:main} and Case VI in \Cref{thm:rep_congruence}. In particular, this applies to modular forms in the Eisenstein subspace $\Ecal_k(\Gamma_1(\ell),\chi)$. In \Cref{lem:q-common_factor}, we showed all the $a_n$'s in \eqref{eqn:max_cong_qexp} are in $\mfrak$. Then its constant term $\frac{B_{k,\chi}}{2k}$ must also be in $\mfrak$ so that there is a non-trivial common factor.
		\end{proof}
		\begin{rem}
			Numerical experiments with  \cite{sagemath} indicate that:
			\begin{itemize}
				\item When $(p-1)\nmid k$, it is possible that $\frac{B_{k,\chi}}{2k}\in \mfrak^s$ for some $s>1$. ($\ell=67, p=11, k\equiv 4\mod (p-1)$)
				\item When $(p-1)\mid k$, $\frac{B_{k,\chi}}{2k}\not\in \mfrak$ for all cases tested. Assuming the maximal congruence $f(q)\equiv 1  \mod \mfrak$ is realized by an element in the Eisenstein space when $(p-1)\mid k$, then we can prove that $\frac{B_{k,\chi}}{2k}\not\in \mfrak$. 
			\end{itemize}
		\end{rem}
	\end{exmp}
	\subsection{Congruence and group cohomology}
	Let $\Zp^{\otimes k}[\chi]$ be the $\Zpx\x \zx{N'}$-representation associated to the character $\Zpx\x\zx{N'}\xrightarrow{(a,b)\mapsto \chi_p(a)\chi'(b)a^k}(\Zp[\chi])^\x$. The maximal congruence of $\Zp^{\otimes k}[\chi]$ as a $\Zpx\x\zx{N'}$-representation is closely related to its group cohomology. Suppose $(R,\mfrak)$ is a $p$-complete discrete valuation ring and let $\varpi\in \mfrak$ be a uniformizer. For a torsion-free $R$-module $M$, the total quotient module $M/\varpi^\infty$ can be defined from a short exact sequence of $R$-modules:
	\begin{equation}\label{eqn:chrres}
	\begin{tikzcd}
	0\rar&M\rar&M[\varpi^{-1}]\rar&M/\varpi^{\infty}\rar&0,
	\end{tikzcd}
	\end{equation}
	where $M[\varpi^{-1}]:=M\otimes_R R[\varpi^{-1}]$. 
	\begin{rem}
		\eqref{eqn:chrres} is the colimit of the following tower of short exact sequences:
		\begin{equation*}
		\begin{tikzcd}
		0\rar& M\rar["\varpi"]\dar[equal]&M\rar\dar["\varpi"]&M/\varpi\rar\dar&0\\
		0\rar& M\rar["\varpi^2"]\dar[equal]&M\rar\dar["\varpi"]&M/\varpi^2\rar\dar&0\\
		0\rar& M\rar["\varpi^3"]\dar[equal]&M\rar\dar["\varpi"]&M/\varpi^3\rar\dar&0\\
		&\cdots&\cdots&\cdots&
		\end{tikzcd}
		\end{equation*}
	\end{rem}
	When $M$ is a $G$-representation in $R$-modules, we note the short exact sequence above is $G$-equivariant. It is straightforward to check:
	\begin{lem}
		Let $G$ be a group and $\rho\colon G\to R^\times$ be a group homomorphism (character), which induces a $G$-action on $R$ as an $R$-module. Denote this representation by $M$. Then either $M$ is a trivial representation (i.e. $M^G=M$), or the maximal congruence of $M$ is given by the an ideal $\mfrak^r$ such that $(M/\varpi^\infty)^G\cong M/\mfrak^r$. 
	\end{lem}
	\begin{lem}\label{lem:delta_injective}
		Assumptions and notation as the above. Suppose $G$ is topologically finitely generated. When $M^G=0$, there is a natural injection $\delta\colon (M/\varpi^\infty)^G\to H_c^1(G;M)$.
	\end{lem}
	\begin{proof}
		Apply $H_c^*(G;-)$ on \eqref{eqn:chrres}, we get a long exact sequence of $G$ cohomology that start with:
		\begin{equation*}
		\begin{tikzcd}
		0\rar& M^G\rar& (M[\varpi^{-1}])^G\rar& (M/\varpi^\infty)^G\rar["\delta"]& H_c^1(G;M)\rar& H_c^1(G;M[\varpi^{-1}])\rar&\cdots
		\end{tikzcd}
		\end{equation*}
		The fixed points $(M[\varpi^{-1}])^G=0$ since \begin{equation*}
		(M[\varpi^{-1}])^G=\left(\colim(M\xrightarrow{\varpi}M\xrightarrow{\varpi}\cdots)\right)^G\cong  \colim(M^G\xrightarrow{\varpi}M^G\xrightarrow{\varpi}\cdots)=(M^G)[\varpi^{-1}]=0.
		\end{equation*} 
		Here we used the fact the finite limit $(-)^G$ (as $G$ is topologically finitely generated) commutes with the filtered colimit $(-)[\varpi^{-1}]$ by  \cite[Theorem 1 in Section IX.2]{MacLane_Categories}. This shows $\delta$ is injective. 	
	\end{proof}
	\begin{prop}\label{prop:cong_gp_coh}
		When $\Zp^{\otimes k}[\chi]$ is a nontrivial $\znx$-representation, i.e. either $k\ne 0$ or $\chi$ is nontrivial, the connecting homomorphism $\delta\colon (\Zp^{\otimes k}[\chi]/\varpi^\infty)^{\Zpx\x\zx{N'}}\to H_c^1(\Zpx\x\zx{N'};\Zp^{\otimes k}[\chi])$ is an isomorphism. 
	\end{prop}
	\begin{proof}
		The injectivity of $\delta$ follows from \Cref{lem:delta_injective}, since the group $\Zpx\x\zx{N'}$ is topologically finitely generated and $\Zp^{\otimes k}[\chi]^{\Zpx\x\zx{N'}}=0$ when $k\ne 0$ or $\chi$ is nontrivial.  
		
		As $p$ is a power of the uniformizer $\varpi$, there are isomorphisms $M[\varpi^{-1}]\cong M[1/p]\cong \Qp^{\otimes k}(\chi)$. Next, we  show that  $H_c^1(\Zpx\x\zx{N'};\Qp^{\otimes k}(\chi))=0$. Note that group cohomology of the finite subgroup $\zpx\times \zx{N'}$ with rational coefficients vanishes in postive degrees. The Hochschild-Serre spectral sequence then implies 
		\begin{equation*}
			H_c^1(\Zpx\x\zx{N'};\Qp^{\otimes k}(\chi))\cong H^0_c(\zpx\times \zx{N'};H^1_c(1+p\Zp; \Qp^{\otimes k}(\chi))).
		\end{equation*}
		The claim then follows from an explicit computation that $H^1_c(1+p\Zp; \Qp^{\otimes k}(\chi)))=0$ when either $k\ne 0$ or $\chi$ is non-trivial. 
	\end{proof}
	Now combining \Cref{thm:main} and \Cref{prop:cong_gp_coh} yields:
	\begin{cor}\label{cor:cong_gp_coh}
		The followings are equivalent:
		\begin{enumerate}
			\item $\Ical\trianglelefteq \Zp[\chi]$ is the maximal congruence of modular forms in $H^0(\Mello(\Gamma_0(N')),\bfomega^{k,\chi})$.
			\item $H_c^1(\Zpx\x\zx{N'};\Zp^{\otimes k}[\chi])\cong  \Zp[\chi]/\Ical$.
		\end{enumerate}
	\end{cor}
	Comparing \Cref{cor:cong_gp_coh} with \Cref{prop:E_kchi_max_congruence} and \Cref{prop:numerator_valuation}, the group cohomology $H_c^1(\Zpx\x\zx{N'};\Zp^{\otimes k}[\chi])$ computes the \emph{denominator} of $\frac{B_{k,\chi}}{2k}\in \Qp(\chi)$ under the assumptions in Cases I-V in \Cref{thm:rep_congruence}.  In Cases VI and VII, this cohomological computation sheds light on the \emph{numerator} of $\frac{B_{k,\chi}}{2k}$ when $(p-1)\nmid k$ (still does not determine the valuation in general).
	
	\begin{rem}
		When the character $\chi$ is trivial, the group cohomology $H_c^1(\Zpx;\Zp^{\otimes k})$ computes the image of the $J$-homomorphism in the stable homotopy groups of spheres. The precisely connection was laid out by the author in \cite{nz_Dirichlet_J}. Therefore we have given a new explanation between congruences of the normalized Eisenstein series $E_{2k}$ of level $1$ and the image of $J$ in this paper. More generally when the character is non-trivial, the author constructed a family of Dirichlet $J$-spectra and $K(1)$-local spheres, whose homotopy groups are computed by the group cohomologies $H_c^s(\Zpx\x\zx{N'};\Zp^{\otimes t}[\chi])$ in the same paper. 
	\end{rem}
	\begin{rem}
		Generalized Bernoulli numbers are also related to the \textbf{Dirichlet $L$-functions} attached to the Dirichlet character $\chi$ \cite{p_adic_L}:  
		\[L(1-k,\chi)=-\frac{B_{k,\chi}}{k}.\]
		When $\chi$ is the trivial character, the Dirichlet $L$-function is the same as the Riemann $\zeta$-function (up to Euler factors at primes dividing $N$). A theorem of Soul\'e \cite{Soule_K-thy_anneaux} implies that the group cohomology $H^1_c(\Zpx;\Zp^{\otimes {2k}})$ is isomorphic to the \'etale cohomology group $H^1_{\et}(\Z[\frac{1}{p}],\Zp(2k))$. The latter computes the $p$-part of the denominator of the special value $\zeta(1-2k)$ of the Riemann $\zeta$-function by the Lichtenbaum Conjecture. This is a special case of the (confirmed) Bloch-Kato Conjecture for the Riemann zeta function. A similar connection between \'etale cohomology and Dirichlet $L$-functions is proved in \cite{Huber-Kings_BK_IMC,Burns-Greither_2003}.
	\end{rem}
	\appendix
	\section{A Riemann-Hilbert correspondence arising from formal groups}\label{Sec:appendix}
	\subsection{Dieudonn\'e modules of formal groups}
	In this subsection, we will compute the Dieudonn\'e modules of the Honda formal groups. First, let's recall the basic definition of Dieudonn\'e modules of formal groups following \cite{Katz_crystalline}. 
	\begin{defn}\label{defn:Dieudonne_mod}
		Let $\kappa$ be a finite field of characteristic $p$ and $\W\kappa$ be its ring of Witt vectors. Let $R$ be a flat $\W\kappa$-algebra such that $R/p$ is an integrally closed domain over $\kappa$. In addition, assume that $R$ admits an endomorphism $\varphi\colon  R\to R$ that lifts the Frobenius $\varphi_0$ on $R/p$ (the $p$-th power map). Let $\Gh_0$ be a formal group over a finite field $R/p$. The \textbf{Dieudonn\'e module} of $\Gh_0$ is a triple $\Dbb(\Gh_0)=(M,F,V)$ consisting of:
		\begin{itemize}
			\item $M=PH^1_{\dR}(\Gh/R)$, where $\Gh$ is a lift of $\Gh_0$ to $R$; $PH^1_{\dR}$ stands for the primitives in first de Rham cohomology. This is explicitly identified as
			\begin{equation}\label{PHdr}
				PH^1_{\dR}(\Gh/R) \cong  \frac{\{f(t)\in \mathbb{Q}_p\otimes R\llb t\rrb\mid f(0)=0,~ df\text{ and } \partial f \text{ are integral}\}}{\{f(t)\in R\llb t\rrb\mid f(0)=0\}}
			\end{equation}
			in \cite[page 193]{Katz_crystalline}, where $\partial f(x,y)=f(x)-f(x+_{\Gh}y)+f(y)$.
			\item $F\colon \varphi^*M\to M$ and $V\colon M\to \varphi^*M$ are induced by the factorization of the $[p]$-series map on $\Gh_0$:
			\begin{equation*}
				\begin{tikzcd}
					{\Gh_0} \arrow[rd] \arrow[rr,"{[p]}"]&&{\Gh_0}\\
					&\varphi_0^*\Gh_0 \arrow[ur]&
				\end{tikzcd}.
			\end{equation*}
		\end{itemize}
	\end{defn}
	\begin{rem}
		$M=PH^1_{\dR}(\Gh/R)$ does not depend on the lift $\Gh_0$ to $R$. This is because the ideal $(p)\trianglelefteq R$ has a divided power structure. Moreover, the assignment $\Gh_0\mapsto PH^1_{\dR}(\Gh/R)$ is functorial in $\Gh_0$.
	\end{rem}
 	To compute the Dieudonn\'e module, we need to simplify \eqref{PHdr}.
	\begin{thm}[{\cite{Cartier_p_typical},\cite[(A2.2.4)]{green}}] 
		Let $\Gh$ be a formal group over a $p$-local algebra $R$. Then $\Gh$ has a coordinate $t$ such that its logarithm has the form
		\begin{equation*}
		\log_{\Gh}(t)=\sum_{i=0}^{\infty}\frac{m_i}{p^i} t^{p^i}, \qquad m_0=1, m_i\in R. 
		\end{equation*}
		This is called the $p$\textbf{-typical} coordinate of $\Gh$. The $p$-series of a $p$-typical formal group satisfies:
		\begin{equation*}
			[p]_{\Gh}(t)=pt+_{\Gh}\sum_{i\ge 1}{}^{\Gh}v_it^{p^i}.
		\end{equation*}
	\end{thm}	
	\begin{rem}
		There are several different, but equivalent, defintions of Dieudonn\'e modules in the literature. The one in \cite{Cartier_p_typical} switches the $F$ and the $V$ maps in \Cref{defn:Dieudonne_mod}. See \cite[\S5.5]{Katz_crystalline} for a comparison of different definitions.
	\end{rem}
	\begin{defn}
		Let $\kappa$ be a perfect field of characteristic $p$, containing $\mathbb{F}_{p^h}$. The Honda formal group $\Gamma_h$ of height $h$ is a one-dimensional commutative formal group scheme over $\kappa$ with a coordinate $t$, such that $[p]_{\Gamma_h}(t)=t^{p^h}$.
	\end{defn}
	\begin{exmp}
		When $h=1$, $\Gamma_1\cong \Gmh$, since the latter has $p$-series $[p](t)=t^p$ over $\Fp$.
	\end{exmp}
	Choose a lift of $\Gamma_h$ to $\W\kappa$ with a $p$-typical coordinate such that $[p]_{\Gamma_h}=pt+_{\Gamma_h}t^{p^h}$.	We can find a basis for $\Dbb(\Gamma_h)$:
	\begin{equation*}
		f_0(t)=\sum_{i\ge 0}\frac{t^{p^{ih}}}{p^i},\quad
		f_1(t)=\sum_{i\ge 0}\frac{t^{p^{ih+1}}}{p^i},\quad\cdots,\quad
		f_{h-1}(t)=\sum_{i\ge 0}\frac{t^{p^{ih+h-1}}}{p^i},
	\end{equation*} such that the matrix representations of $F$ and $V$ with respect to this basis $\{f_0,\cdots,f_{h-1}\}$ are:
	\begin{align*}
		F=\begin{pmatrix}
			&pI_{h-1}\\
			1&
		\end{pmatrix},&&V=\begin{pmatrix}
			&p\\
			I_{h-1}&
		\end{pmatrix}.
	\end{align*}
	In particular, we have $V^h(f_0)=f_0\left(t^{p^h}\right)=pf_0(t)$. When $h=1$, we have $F(f_0)=f_0$ and $V(f_0)=pf_0$.
	\subsection{Statement of the correspondence}	
	\begin{thm}\label{thm:rhc}
		Let $R$ be as in \Cref{defn:Dieudonne_mod}. Then the following categories are equivalent:
		\begin{enumerate}
			\item \Dieudonne modules $(M,F,V)$ such that $M/V$ is an invertible $R/p$-module generated by $\gamma\in M$ such that
			\[F\gamma\equiv V^{h-1}(a_0\gamma) \mod V^{h}\text{, where }a_0\in(R/p)^\x.\]
			\item $\piet(R)$-representations in rank $1$ free $\mathcal{O}_h$-modules, where
			\[\mathcal{O}_h=\W\mathbb{F}_{p^h}\langle\sigma\rangle/\{\sigma^h=p, a^{\varphi}\sigma=\sigma a \}.\]
			\item One dimensional formal groups of height $h$ over $R/p$.
		\end{enumerate}
	\end{thm}
	\begin{proof}
		We prove the equivalence as follows:
		\begin{description}
			\item[(1)$\iff$(3)] (1) is the description of the Dieudonn\'e module of a height $h$ formal group. The equivalence was proved in \cite{de_Jong_crystalline_Dieudonne}.
			\item[(2)$\iff$(3)] $\Ocal_h$ is the algebra of endomorphisms of $\Gamma_h$ over $\Fbb_{p^h}$. The equivalence follows from the theory of Galois descent for formal groups and Lazard's result \cite[Théorème IV]{Lazard_1955} that formal groups of the same height are \'etale locally isomorphic to each other.\qedhere
		\end{description}
	\end{proof}
	The Riemann-Hilbert correspondence (1)$\iff$(2) is then the equivalence of the Dieudonn\'e module data and the Galois descent data of formal groups. When $h=1$, we recover Katz's \Cref{thm:padic41Wk}. Congruences of the categories in this equivalence are related to the finite subgroup schemes of the formal groups.
	\begin{thm}\label{thm:rhc_cong}
		Suppose $(M,F,V)$, $\rho\in H_c^1(\piet(R);\mathcal{O}_h^\x)$ and $\Gh$ correspond to each other in \Cref{thm:rhc}. Then the followings are equivalent:
		\begin{enumerate}
			\item There is a generator $\gamma\in M$ such that $F\gamma\equiv V^{h-1}\gamma\mod V^{h-1+m}$.
			\item The Galois representation $\rho$ is trivial mod $\sigma^m$.
			\item The finite subgroup scheme of $\Gh$ of rank $p^m$ is isomorphic to the corresponding rank $p^m$ finite subgroup scheme of the Honda formal group $\Gamma_h$ of height $h$. 
		\end{enumerate}
	\end{thm}
	\begin{proof}
		The Dieudonn\'e modules and Galois descent data for finite subgroup schemes of $\Gh_0$ are described in \Cref{prop:fin_gp_d-mod} and \Cref{prop:fin_gp_gal_desc}, respectively. Using the computation of $\Dbb(\Gamma_h)$ at the end of the previous subsection, the proof of the theorem is now similar to that of \Cref{thm:rhc_formal_A-mod} in the height $1$ case.
	\end{proof} 
	We recover \Cref{thm:padic41} when $h=1$. \Cref{thm:rhc} and \Cref{thm:rhc_cong} also hold for $p$-divisible formal $A$-modules of finite dimensions in general.
	\printbibliography
\end{document}